\documentclass[a4paper,11pt]{amsart}
\usepackage{amsfonts,amssymb,amsthm,cite,amsmath,amstext}
\usepackage{color}
\usepackage[colorlinks,linkcolor=black,citecolor=black]{hyperref}
\usepackage{comment}
\usepackage{framed}

\usepackage{fancyhdr}
\definecolor{shadecolor}{gray}{0.875}

\newtheorem{thrm}{Theorem}[section]
\newtheorem{thrmx}{Theorem}

\newtheorem{lem}[thrm]{Lemma}
\newtheorem{cor}[thrm]{Corollary}
\newtheorem{prop}[thrm]{Proposition}
\newtheorem{conj}[thrm]{Conjecture}

\theoremstyle{definition}
\newtheorem{defn}[thrm]{Definition}

\newtheorem{exmple}[thrm]{Example}
\newtheorem{rmk}[thrm]{Remark}

\DeclareMathOperator{\val}{Val}

\DeclareMathOperator{\vol}{vol}

\DeclareMathOperator{\GL}{GL}
\DeclareMathOperator{\Gr}{Gr}
\DeclareMathOperator{\supp}{supp}

\DeclareMathOperator{\Aut}{Aut}
\DeclareMathOperator{\reldeg}{reldeg}

\DeclareMathOperator{\reld}{reld}

\DeclareMathOperator{\ev}{ev}
\DeclareMathOperator{\Val}{Val}
\DeclareMathOperator{\SO}{SO}

\DeclareMathOperator{\Id}{Id}
\DeclareMathOperator{\even}{even}
\DeclareMathOperator{\odd}{odd}
\DeclareMathOperator{\Vect}{Vect}

\title{Positivity of valuations on convex bodies and invariant
valuations by linear actions}
\author{Nguyen-Bac Dang and Jian Xiao}
\date{}

\thanks{The first author is supported by the ERC-starting grant project ``Nonarcomp'' no.307856, and by the brazilian project ``Ci\^encia sem fronteiras'' founded by the CNPq.
}

\begin{document}
\maketitle

\begin{abstract}
In this paper, we endow the  space of continuous translation invariant valuation on convex sets generated by mixed volumes coupled with a suitable Radon measure on tuples of convex bodies with two appropriate norms.
This enables us to construct a continuous extension of the convolution operator on smooth valuations to non-smooth valuations, which are in the completion of the spaces of valuations with respect to these norms.

The novelty of our approach lies in the fact that our proof does not rely on the general theory of wave fronts, but on geometric inequalities deduced from optimal transport methods.

We apply this result to prove a variant of Minkowski's existence theorem, and generalize a theorem of Favre-Wulcan and Lin in complex dynamics over toric varieties by studying the linear actions on the Banach spaces of valuations and by studying their corresponding eigenspaces.
%
\end{abstract}

\tableofcontents

\section*{Introduction}

Let $E$ be a Euclidian real vector space of dimension $n$, and let  $\mathcal{K}(E)$ be the family of convex bodies (i.e., compact closed convex subsets) of $E$.
We endow the space $\mathcal{K}(E)$ with the Hausdorff metric which is defined by
\begin{equation*}
  d_H (K, L)=\min\{\varepsilon>0|K\subset L+\varepsilon\mathbf{B}\ \&\ L\subset K+\varepsilon\mathbf{B}\},
\end{equation*}
for any $K,L \in \mathcal{K}(E)$, where $\mathbf{B}$ is the unit ball in $E$.
A real (convex) valuation $\phi$ on $E$ is a function $\phi: \mathcal{K}(E) \to \mathbb{R}$ such that
\begin{equation*}
\phi(K\cup L) = \phi (K) + \phi(L) - \phi(K\cap L)
\end{equation*}
for any $K, L \in \mathcal{K}(E)$ satisfying $K\cup L \in \mathcal{K}(E)$.
Moreover, a valuation $\phi$ is called \emph{translation invariant} if $\phi(K + t) = \phi(K)$ for any $K \in \mathcal{K}(E)$ and any $t \in E$, and it is called \emph{continuous} if it is continuous with respect to the topology of $\mathcal{K}(E)$ given by the metric $d_H$.
We denote by $\Val(E)$ the Banach space of continuous, translation invariant valuations on $E$ with the norm of $\phi \in\val(E)$ given by:
\begin{equation}\label{eq banach Val}
|| \phi || := \sup_{K\subset \mathbf{B}} |\phi(K)|,
\end{equation}
where
the supremum is taken over all convex bodies $K$ contained in the unit ball $\mathbf{B}$.


The most basic examples of homogeneous valuations are
given by maps of the form
 $$K \mapsto V(L_1, \ldots , L_{n-i}, K[i]),$$
where $L_1, \ldots, L_{n-i} \in \mathcal{K}(E)$ and the symbol $V(-)$ denotes the mixed volume of convex bodies, and $K[i]$ means that the convex body $K$ is repeated $i$ times in the expression of the mixed volume.

A priori, the space $\Val(E)$ is merely a subspace of the space of continuous functions on the metric space $\mathcal{K}(E)$. However, it carries surprisingly rich algebraic structures, which play an important role in integral geometry, more precisely in the study  of projection bodies (see e.g. \cite{petty_projection,lutwak_rotation,bourgain_lindenstrauss,lutwak_inequalities_projection,ludwig,abardia_bernig})
often applied in Crofton type formulas (see \cite{busemann}), of additive kinematic formulas and intersectional kinematic formulas (see \cite{bernigConvolution}), of hermitian integral geometry (see \cite{alesker_lefschetz,bernigfuHermitian}) and more general kinematic formulas on manifolds with corners (see \cite{bernig_brocker,aleskerbernigJDG17}).
Further notions of valuations can be inferred, particularly valuations with values in an arbitrary vector space are natural extension of the above definition and proved to be quite fruitful (see \cite{wannerer_equivariant,schuster_wannerer_contravariant,parapatits_schuster,
 wannerer_integral_unitary_area,haberl_parapatits,
 bernig_hug_integral_tensor,schuster_wannerer_minkowski}).
A first result indicating these structures is McMullen's  theorem (see \cite{McMullenDecomp}), which states that the space of valuations can be decomposed into:
\begin{equation*}
  \Val(E)= \bigoplus_{i=0} ^n \Val_i (E).
\end{equation*}
where $\Val_i(E)$ denotes the subspace of $\Val(E)$ containing homogeneous valuations of degree $i$ (i.e valuations $\phi$ for which $\phi(\lambda K) = \lambda^i \phi(K)$ for all $\lambda >0$ and all $ K\in \mathcal{K}(E)$).

One of the cornerstones of modern integral geometry is Alesker's irreducibility theorem, which states that mixed volumes span a dense subset in $\Val(E)$. The proof of this deep result, due to Alesker (see \cite{aleskermcmullenconj}), uses results from the theory of $\mathcal{D}$-modules and representation theory.
This breakthrough shed light upon new perspectives in the field. Alesker, then Bernig-Fu, constructed two algebraic operations, a product and a convolution for a  particular class of valuations called \textit{smooth valuations}.
Let us recall their definition.
The Lie group $\GL(E)$ has a natural action on $\Val(E)$:
\begin{align*}
  &\GL(E) \times \Val(E) \rightarrow \Val(E), \\
  &(g, \phi)\mapsto g\cdot \phi,
\end{align*}
where $g\cdot \phi(K): = \phi(g^{-1}K)$ for any $K\in \mathcal{K}(E)$ (see \cite{aleskermcmullenconj}) and a valuation $\phi$ is called \emph{smooth} if the map $g\mapsto g\cdot\phi$ is smooth.  We denote by $\Val^\infty(E)$ the subset of $\Val(E)$ of smooth translation invariant valuations.

One of the main feature of this subspace is that there is a
convolution operation on it, uncovered by Bernig-Fu (see \cite{bernigConvolution}) and studied further by Alesker (see \cite{alesker_fourier_type_transform}).
They proved that there exists a unique continuous, symmetric bilinear map $* $ which is homogeneous of degree $-n$:
\begin{align*}
&\Val^\infty(E) \times \Val^\infty(E) \to \Val ^\infty(E),\\
&(\phi, \varphi)\mapsto \phi * \varphi,
\end{align*}
such that for any $K, L \in \mathcal{K}(E)$ with smooth and strictly convex boundary, one has that:
\begin{equation*}
\vol( \cdot + K ) * \vol( \cdot  + L) = \vol( \cdot + K + L) \in \Val^\infty(E).
\end{equation*}
In particular, assume that $K_1, ...,K_{n-i}, L_1, ...,L_{n-j} \in \mathcal{K}(E)$ have smooth and strictly convex boundary, then
\begin{equation}\label{eq smooth conv}
V(-; K_1, \ldots , K_{n-i}) * V(-; L_1, \ldots , L_{n-j}) = \frac{i!j!}{n!} V(-;K_1, \ldots , K_{n-i}, L_1, \ldots , L_{n-j}).
\end{equation}

Our goal is to extend further the convolution operation to more general valuations using ideas coming from complex geometry. To do so, we
introduce the following key notion on valuations and define two suitable norms on this space.
For any positive Radon measure $\mu$ on $\mathcal{K}(E)^{n-i}$ such that
\begin{equation*}
\int_{\mathcal{K}(E)^{n-i}} V( \mathbf{B} [i], K_1,\ldots, K_{n-i}) d\mu(K_1, \ldots , K_{n-i}) < + \infty,
\end{equation*}
we define a valuation $\phi_\mu$ given by
\begin{equation*}
\phi_\mu (L) = \int_{\mathcal{K}(E)^{n-i}} V( L [i], K_1,\ldots, K_{n-i}) d\mu(K_1, \ldots , K_{n-i}).
\end{equation*}
Observe that the dominated convergence theorem ensures the fact that $\phi_\mu$ is a continuous translation invariant valuation. Moreover, such a valuation is monotone in the sense that: if $K \subset L \in \mathcal{K}(E)$ then $\phi(K) \leqslant \phi(L)$.
Note that the linear map $\mu \rightarrow \phi_\mu$ may not be injective.

A valuation $\phi \in \Val_i(E)$ is said to be \emph{$\mathcal{P}$-positive}, if there exists a measure $\mu$ as above such that $\phi = \phi_\mu$. We denote by $\mathcal{P}_i \subset \Val_i(E)$ the set of $\mathcal{P}$-positive homogeneous valuations of degree $i$.
As an example, the positive linear combinations of mixed volumes of degree $i$ are contained in $\mathcal{P}_i$.
There is a close analogy between these cones with the convex cones of positive cohomology classes in complex geometry. For example, the analog of $\mathcal{P}_{n-1}$ corresponds to the cone of big movable divisor classes, and the the analog of $\mathcal{P}_{1}$ corresponds to the cone of movable curve classes (see \cite{lehmann2016correspondences}).

\begin{rmk}
Let us emphasize that the  cone $\mathcal{P}_i$ introduced above is contained in the cone of positive valuations containing all the valuations $\phi \in \Val_i(E)$, which are non-negative on every convex body (see \cite{PWposval,bernigfuHermitian}).
\end{rmk}

By a polarization argument, a valuation $\phi \in \mathcal{P}_i$ defines a unique function on $\mathcal{K}(E)^i$:
\begin{equation*}
\phi(L_1, \ldots, L_i) = \frac{1}{i!}\left ({\dfrac{\partial^i}{\partial t_1 \partial t_2 \ldots \partial t_i}} \right )_{| t_1 = \ldots = t_i = 0^+} \phi( t_1 L_1 + \ldots + t_i L_i) ,
\end{equation*}
where $L_1, \ldots, L_i$ are convex bodies. If $L_1 =...=L_i =L$, then $\phi(L_1, \ldots, L_i) = \phi(L)$.


The convex cone $\mathcal{P}_i$ generates a vector space $\mathcal{V}_i' = \mathcal{P}_i - \mathcal{P}_i \subset \Val_i (E)$.
Precisely, any valuation $\phi \in\mathcal{V}_i'$ is of the form $\phi_\mu$ where $\mu$ is a signed Radon measure on $\mathcal{K}(E)^{n-i}$ such that its absolute value $|\mu|$ satisfies:
\begin{equation*}
\int_{\mathcal{K}(E)^{n-i}} V(\mathbf{B}[i] , K_1, \ldots , K_{n-i}) d|\mu|(K_1, \ldots , K_{n-i} ) < +\infty.
\end{equation*}
We expect that this space contains all smooth valuations however we obtained only a partial inclusion for even valuations (see Proposition \ref{smooth valuation inclusion} and Proposition \ref{odd val inclusion}) ).

We shall define two appropriate norms on the vector space generated by this cone.

\begin{defn}
For any $\phi \in \mathcal{V}'_i$, the norm $||\cdot ||_{\mathcal{P}}$ is defined by
\begin{equation*}
  ||\phi ||_{\mathcal{P}} := \inf\{t\geq 0|\ |\phi(L_1,...,L_i)|\leq t V(\mathbf{B}[n-i], L_1, ...,L_i)\ \textrm{for any}\ L_1, ...,L_i \in \mathcal{K}(E)\}.
\end{equation*}
\end{defn}


The finiteness of $||\phi||_\mathcal{P}$  follows from inequalities between mixed volumes referred as ``reverse Khovanskii-Teissier inequalities'' (see \cite{lehmann2016correspondences} or Theorem \ref{thrm rev KT}), which were  obtained by Lehmann-Xiao using mass transport estimates.

The idea behind this construction comes from the study of pseudo-effectivity of cohomology classes on complex projective spaces in complex geometry, which can then be explicitly related to polytopes when one specializes in classes over toric varieties.
An important property of this norm is that the subspace of smooth valuations forms a dense subspace in $\mathcal{V}'_i$ with respect to this norm (see Theorem \ref{thrm_density_smooth}).

The second norm is induced by the positive cone $\mathcal{P}_i$.

\begin{defn}
For any $\phi \in \mathcal{V}_i'$, the norm $||\phi||_\mathcal{C}$ is given by the following formula:
\begin{equation*}
  ||\phi||_\mathcal{C} :=\inf_{\phi = \phi_{+}-\phi_{-}, \phi_{\pm}\in\mathcal{P}_i} (\phi_{+} (\mathbf{B})+\phi_{-} (\mathbf{B})),
\end{equation*}
\end{defn}
Compared to the previous norm, we do not know whether smooth valuations are dense in $\mathcal{V}_i'$ for the topology induced by this cone norm. Nevertheless, this norm plays a crucial role in the present paper.

Take $\mathcal{V}_i ^\mathcal{P}$ and $\mathcal{V}_i ^\mathcal{C}$ to be the completion of $\mathcal{V}_i'$ with respect to the norm $|| \cdot ||_{\mathcal{P}}$ and $|| \cdot ||_\mathcal{C}$ respectively.
By definition, for any $L \subset \mathbf{B}$ we have $|\phi(L)|\leq \vol(\mathbf{B})||\phi ||_{\mathcal{P}}$, hence $||\phi||\leq \vol(\mathbf{B})||\phi ||_{\mathcal{P}}$. Actually, there is a sequence of continuous injections (see Corollary \ref{cor injections}):
\begin{equation*}
(\mathcal{V}' _i,|| \cdot ||_{\mathcal{C}}) \hookrightarrow ( \mathcal{V}' _i, || \cdot ||_{\mathcal{P}}) \hookrightarrow (\Val_i(E) , || \cdot || ) .
\end{equation*}
As a consequence of Alesker's irreducibility theorem, the spaces $\mathcal{V}_i^\mathcal{P}$ and $\mathcal{V}_i^\mathcal{C}$ are dense in $(\Val_i(E) , || \cdot || )$.

Set $\mathcal{V}^\mathcal{C}= \oplus_{i=0}^n \mathcal{V}_i^\mathcal{C}$ and $\mathcal{V}^\mathcal{P} = \oplus_{i=0}^n \mathcal{V}_i^\mathcal{P}$.
Our first theorem shows that the convolution of valuations can be extended continuously to the space $\mathcal{V}^\mathcal{C}\times \mathcal{V}^\mathcal{P}$ (see Theorem \ref{thrm extension *}).

\begin{thrmx}\label{thrmx conv *}
Fix two integers $i,j$ such that $ 2 n \geqslant i+j \geqslant n $, then the following properties are satisfied.
\begin{enumerate}
\item The convolution $*: (\mathcal{V}'_i \cap \Val^\infty(E)) \times  (\mathcal{V}'_j \cap \Val^\infty(E)) \rightarrow \mathcal{V}'_{i+j -n}$ extends continuously to a  bilinear map
\begin{align*}
  *: & \mathcal{V}^\mathcal{C}_i \times \mathcal{V}^\mathcal{C} _j    \rightarrow \mathcal{V}^\mathcal{C} _{i+j -n}\\
  & (\Phi, \Psi)\mapsto \Phi * \Psi.
\end{align*}
Thus $\mathcal{V}^\mathcal{C}$ has a structure of graded Banach algebra with unit given by $\vol$.
\item The convolution $*: (\mathcal{V}'_i \cap \Val^\infty(E)) \times  (\mathcal{V}'_j \cap \Val^\infty(E))\rightarrow \mathcal{V}'_{i+j -n}$ extends continuously to a  bilinear map
\begin{align*}
  *: & \mathcal{V}^\mathcal{C}_i \times \mathcal{V}^\mathcal{P} _j    \rightarrow \mathcal{V}^\mathcal{P} _{i+j -n}\\
  & (\Phi, \Psi)\mapsto \Phi * \Psi.
\end{align*}
Thus the space $\mathcal{V} ^\mathcal{P}$ has a structure of graded $\mathcal{V}^\mathcal{C}$-module.
\end{enumerate}
\end{thrmx}

A priori, the convolution is only well defined on the space of smooth valuations $\Val ^\infty(E)$, and one cannot extend it continuously to $\Val(E)$. Theorem \ref{thrmx conv *} proves that the convolution is defined if we consider finer topologies than the one in $\Val_i(E)$, completely on $\mathcal{V}^\mathcal{C}$ and on $\mathcal{V}^\mathcal{C}\times \mathcal{V}^\mathcal{P}$.
It is thus natural to ask whether the convolution can be extended further to $\mathcal{V}^\mathcal{P}\times \mathcal{V}^\mathcal{P}$.



Observe that in statement of Theorem \ref{thrmx conv *}, we obtain an extension of the convolution for valuations which belong to the intersection $\mathcal{V}' \cap \Val^\infty(E)$. The reason is  that only the inclusion of even smooth valuations in $\mathcal{V}'$ is known (see Proposition \ref{smooth valuation inclusion}). However, we expect $\mathcal{V}'$ to contain all smooth valuations (see Proposition \ref{odd val inclusion}).

Alesker-Bernig (see \cite{aleskerGeneralProduct}, \cite{aleskerbernigJDG17}) studied another extension on the space of generalized valuations (the dual of $\Val^\infty(E)$) satisfying particular conditions, using the general theory of wave fronts. The main difference with our approach is that the convolution of a valuation in $\mathcal{V}^\mathcal{C}$ with a valuation in $\mathcal{V}^\mathcal{C}$ or $\mathcal{V}^\mathcal{P}$ is well-defined, whereas it is not clear whether the convolution makes sense in the space of generalized valuations since the criterion on the wave front given by Alesker-Bernig may not be  satisfied.

This extension can be understood explicitly as follows.
Using \eqref{eq smooth conv}, we observe that if
$\mu $ and $\nu$ are two Radon measures on $\mathcal{K}(E)^{n-i}$ and $\mathcal{K}(E)^{n-j}$ respectively so that their associated valuations $\phi_\mu $ and $\phi_\nu $ belong to $\mathcal{V}_i' $ and $\mathcal{V}_j'$ respectively, then the valuation $\phi_\mu *  \phi_\nu \in \mathcal{V}_{i+j-n} ' $ is a valuation associated to the measure:
\begin{equation*}
\dfrac{i! j!}{n!} p_1^* \mu \otimes p_2^* \nu,
\end{equation*}
where $p_1: \mathcal{K}(E)^{2n-i-j} \to \mathcal{K}(E)^{n-i}$ and $p_2 : \mathcal{K}(E)^{2n-i-j} \to \mathcal{K}(E)^{n-j}$ are the projections onto the first $n-i$ factors and the last $n-j$ factors respectively.
The formula for the valuation $\phi_\mu* \phi_\nu$ is given by:
\begin{align*}
\phi_\mu * \phi_\nu (-)
:= \dfrac{i! j!}{n!} \int_{\mathcal{K}(E)^{2n-i-j}} V(-; K_1, \ldots , K_{n-i}, K_1' ,\ldots K'_{n-j} ) d\mu(K_1, \ldots , K_{n-i}) d\nu(K_1' , \ldots , K'_{n-j} ),
\end{align*}
which is always well defined by Proposition \ref{lem * well_defined}.

\bigskip

A $\mathcal{P}$-positive valuation looks complicated at first glance, in this paper we prove a structure theorem on these valuations.

Let $L_1,...,L_{n-1} \in \mathcal{K}(E)$ be convex bodies with non-empty interior, by Minkowski's existence theorem (see \cite{alexandrov1938}), there exists a unique (up to a translation) convex body $L \in \mathcal{K}(E)$ with non-empty interior such that
\begin{equation*}
V(L_1, \ldots , L_{n-1}, - )=V(L[n-1],-).
\end{equation*}

Our next result can be considered as a variant of Minkowski's existence theorem (see Theorem \ref{thrm general mink} and Proposition \ref{prop compct}).
We say that a valuation $\phi \in\mathcal{P}_i$ is \emph{strictly $\mathcal{P}$-positive} if there exists $\epsilon>0$ such that
\begin{equation*}
\phi (L_1, \ldots , L_i) \geqslant \epsilon V(\mathbf{B} [n-i],L_1, ..., L_i)
\end{equation*}
holds for any convex bodies $L_1, \ldots, L_i$.

\begin{thrmx}\label{thrmx general mink intr}
For any $\psi \in \mathcal{P}_i$ strictly $\mathcal{P}$-positive, there is a constant $c>0$ (depending only on $\psi$) and a convex body $B$ with $\vol(B)=1$ such that
\begin{equation*}
  \psi *V(B [i-1], -  )=cV(B[n-1],-) \in \Val_1(E).
\end{equation*}
Moreover, up to translations the solution set
\begin{equation*}
S=\{B \in \mathcal{K}(E)| \psi* V(B [i-1], -  )=cV(B[n-1];-), \vol(B)=1\}
\end{equation*}
is compact in $\mathcal{K}(E)$ endowed with the Hausdorff metric.
\end{thrmx}

Observe that when $i=1$, the above result is just a consequence of Minkowski's existence theorem \cite{alexandrov1938,schneider_convex} (see Example \ref{sec exmple}).

\bigskip

Our next results focus on linear actions on valuations.
Recall that the group $\GL(E)$ has a natural action on $\Val(E)$.
We are interested in the behaviour of the sequence $\{g^k \cdot \phi\}_{k=1} ^\infty$ where $\phi \in \mathcal{V}_{n-i}^\mathcal{P}$ or $\mathcal{V}_{n-i}^\mathcal{C}$ and $g \in \GL(E)$.
Given $g\in \GL(E)$, $\phi \in \mathcal{P}_{n-i} $ and $\psi \in \mathcal{P}_{i}$ two strictly $\mathcal{P}$-positive valuations, we define the \emph{$i$-th dynamical degree} of $g$ by
\begin{align*}
  d_{i} (g) : =\lim_{k\rightarrow \infty} ((g^k\cdot \phi)*\psi)^{1/k}.
\end{align*}
The terminology ``dynamical degree'' comes from the study of dynamics of holomorphic maps, where these numbers are defined for rational self-maps on projective varieties.
These two notions of dynamical degrees are closely related in the particular case of rational self-maps over toric varieties which preserve the torus action.

Note that $g$ induces a linear operator (denoted by $g_{n-i}$) on the Banach spaces $(\mathcal{V}_{n-i} ^\mathcal{P} , ||\cdot||_{\mathcal{P}})$ (respectively, $(\mathcal{V}_{n-i} ^\mathcal{C} , ||\cdot||_{\mathcal{C}})$).
A direct application of the mixed volumes estimates (see Theorem \ref{thrm rev KT}) and the method in \cite{dang2017degrees} shows that the number $d_i (g)$ is well-defined and is equal to the norm of the operator $g_{n-i}$.
We also remark that the norm of the linear operator $g_{n-i}$ acting on $\mathcal{V}_{n-i}^\mathcal{P}$ and $\mathcal{V}_{n-i}^\mathcal{C}$ are equal.
Our next theorem (see Theorem \ref{thrm norm degree} and Theorem \ref{thrm dydeg formula}) relates the norm of $g_{n-i}$, the eigenvalues of $g$ and the dynamical degrees.

\begin{thrmx}\label{thrmx dydegree intro}
Given $g\in \GL(E)$, the dynamical degree $d_{i} (g)$ exists and is independent of the choices of the strictly $\mathcal{P}$-positive valuations $\phi \in \mathcal{P}_{n-i} , \psi\in \mathcal{P}_{i} $.
Moreover, assume that $\rho(g_{n-i})$ is the spectral radius of $g_{n-i}$ and $\rho_1, ...,\rho_n$ are the eigenvalues of $g$ satisfying
\begin{equation*}
  |\rho_1|\geq |\rho_2|\geq...\geq |\rho_n|,
\end{equation*}
then the $i$-th dynamical degree $d_{i}(g)=\rho(g_{n-i})= |\det{g}|^{-1} \prod_{k=1} ^{i} |\rho_k|$.
\end{thrmx}

Our proof relies on the observation that the dynamical degrees define continuous mappings from $\GL(E)$ to $\mathbb{R}$. We are then reduced to prove the Theorem \ref{thrmx dydegree intro} for diagonalizable matrices.
Observe that our proof gives an alternative approach to the results of Lin (see \cite[Theorem 6.2]{lin_algebraic_stability_and_degree_growth}) and Favre-Wulcan (see \cite[Corollary B]{favrewulcandegree}) which relied on Minkowski weights and integral geometry respectively.

We say that a valuation $\phi$ is \emph{$d_i(g)$-invariant} if it belongs to the eigenspace of eigenvalue $d_i(g)$ (i.e., $g\cdot \phi = d_{i}(g) \phi$).

By Alexandrov-Fenchel inequality or Theorem \ref{thrmx dydegree intro}, it is clear that the sequence of dynamical degrees $\{d_i (g)\}$ is log-concave. In particular, $d_i (g) ^2 \geq d_{i+s} (g)d_{i-s} (g)$.
Our last theorem (see Theorem \ref{thrm pos invariant}) gives some positivity properties of invariant valuations under a natural strict log-concavity assumption on these numbers.

\begin{thrmx}\label{thrmx pos invariant intro}
Assume that $2i\leq n$, and $g\in \GL(E)$.
Then the following properties are satisfied.
\begin{enumerate}
\item The subspace of $d_i(g)$-invariant valuations in $\Val_{n-i}(E)$ is non trivial.
\item Assume that
the strict log-concavity inequality is satisfied for $s \leqslant \min(i, n-i)$:
\begin{equation*}
  d_{i}(g)^2 > d_{i-s}(g) d_{i+s}(g),
\end{equation*}
then for any two $d_i(g)$-invariant valuations $\psi_1 \in \mathcal{V}_{n-i} ^\mathcal{P}, \psi_2 \in \mathcal{V}_{n-i} ^\mathcal{C}$, we  have
\begin{equation*}
  \psi_1 * \psi_2  =0.
\end{equation*}

\item Assume that
\begin{equation*}
d_1^2(g) > d_2(g),
\end{equation*}
then there exists a unique (up to a multiplication by a positive constant) $d_1(g)$-invariant  valuation $\psi\in \overline{{\mathcal{P}}_{n-1}}$ in the closure of ${\mathcal{P}}_{n-1}$ in $\mathcal{V}_{n-1}^\mathcal{P}$. Moreover, $\psi$ lies in an extremal ray of $\overline{\mathcal{P}_{n-1}}$.
\end{enumerate}
\end{thrmx}

\begin{rmk}
In the statements (2) and (3), one could also replace the space $\mathcal{V}^\mathcal{P}$ by $\mathcal{V}^\mathcal{C}$.
\end{rmk}

In the study of monomial maps, the conclusion of $(3)$ implies also the existence of a unique invariant $b$-divisor class in the sense of \cite{favrewulcandegree}.
The results $(2)$ and $(3)$ can be understood as the higher dimensional convex analog of a result by  \cite{boucksom2008degree} for projective surfaces.
Given a projective surface $X$ and a dominant rational map $f$ on it.
Suppose that the dynamical degree $d_1(f)$ and $d_2(f) $ satisfy $d_1(f)^2> d_2(f)$, Boucksom, Favre and Jonsson proved the existence and the uniqueness (up to scaling) of two nef Weil-classes $\theta^+$ and $\theta^-$ which are $d_1(f)$-invariant by $f^*$ and $f_*$ respectively. They proved also that the self-intersection  $\theta^+ \cdot \theta^+  $ is equal to zero.

\subsection*{Organization}
In Section \ref{sec preliminary}, we give a brief review of valuations on convex sets. Section \ref{sec positive} is devoted to the study of the cone $\mathcal{P}_i$, and to  the continuous extension of the convolution operator. In Section \ref{sec minkowski}, we prove Theorem \ref{thrmx general mink intr}.
Finally, we use our first result to study the linear actions on valuations and the positivity properties of invariant valuations under a natural strict log-concavity assumption on the dynamical degrees in Section \ref{sec dynamical degree}.

\subsection*{Acknowledgements}
We would like to thank  S.~Boucksom and C.~Favre for their interests and comments regarding this paper.
We are particularly grateful to A.~Bernig who suggested many improvements of our first result, to Thomas Wannerer for pointing out a mistake in our previous version.
We would also like to thank S.~Alesker for answering several questions on his works on the convolution of valuations. The first author would also like to thank L.~DeMarco for supporting his stay in Northwestern University to work on this project.

\section{Preliminaries}\label{sec preliminary}

\subsection{Valuation on convex sets}
We first give a brief overview on the theory of valuations on convex sets. The classical references are \cite{McMuSchValuation, McMuDissections,schneider_convex} and we also refer the reader to the more recent surveys \cite{aleskerSurvey}, \cite{AleskerFuBook} and \cite{bernig2012algebraic}.

Let $E$ be a Euclidian real vector space of dimension $n$. We denote the family of non-empty compact convex subsets of $E$ by $\mathcal{K}(E)$. The set $\mathcal{K}(E)$ has a natural topology induced by the Hausdorff metric defined as follows:
\begin{equation*}
  d_H (K,L):=\inf\{\varepsilon>0|\ K \subset L + \varepsilon \mathbf{B}\ \& \ L \subset K + \varepsilon \mathbf{B}\},
\end{equation*}
where $\mathbf{B}$ is the unit ball, where $K,L \in \mathcal{K}(E)$ and where $+$ is the Minkowski sum. Observe that Blaschke's selection theorem implies that the metric space $(\mathcal{K}(E), d_H)$ is locally compact. Moreover, for each convex body, one can associate it to its support function. This map induces an isometry between $(\mathcal{K}(E), d_H)$ and the space of continuous functions $C^0 (\mathbb{S}^{n-1})$ endowed with $L^\infty$-norm, where $\mathbb{S}^{n-1}$ is the $n-1$ dimensional sphere.

\begin{defn}
A functional $\phi: \mathcal{K}(E) \rightarrow \mathbb{R}$ is called a (real) \emph{valuation} if
\begin{equation*}
  \phi(K\cup L)= \phi(K) + \phi(L) - \phi(K\cap L)
\end{equation*}
whenever $K, L, K\cup L \in \mathcal{K}(E)$.
\end{defn}

\begin{rmk}
The valuations we consider are valuations on convex sets. Although the valuation theory has been extended to non necessarily convex sets on manifolds (see \cite{aleskerSurvey}), we shall follow the classical terminology.
\end{rmk}

\begin{defn}
A valuation on convex sets $\phi$ is called \emph{continuous} if $\phi$ is continuous with respect to the Hausdorff metric $d_H$; A  valuation $\phi$ is called \emph{translation-invariant} if $\phi(K + x) = \phi(K)$ for any $K\in \mathcal{K}(E)$ and any $x \in E$.
\end{defn}

Let us denote by $\Val(E)$ the space of translation-invariant continuous valuations. The linear space $\Val(E)$ has the natural topology given by a sequence of semi-norms:
\begin{equation*}
  ||\phi||_N = \sup_{K \subset \mathbf{B}_N} |\phi (K)|,
\end{equation*}
where $\mathbf{B}_N$ is the ball of radius $N$. This sequence of semi-norms defines a Fr\'{e}chet space structure on $\Val (E)$. Actually, $\Val (E)$ is a Banach space endowed with the norm $||\cdot||_1$, which we denote by $|| \cdot ||$.

\subsection{McMullen's grading decomposition}
We recall McMullen's decomposition of the space of valuations $\Val(E)$.
Fix a non-negative real number $\alpha$.
A  valuation $\phi$ on convex sets  is called \textit{$\alpha$-homogeneous} if $\phi(\lambda K) = \lambda^\alpha \phi(K)$ for any $\lambda \geq 0, K \in \mathcal{K}(E)$.
We denote by $\Val_{\alpha} (E)$ the subspace of $\Val(E)$ of $\alpha$-homogeneous valuations on convex sets. The following result is due to McMullen \cite{McMullenDecomp}.

\begin{thrm}[McMullen decomposition]
Let $n = \dim E$, then
\begin{equation*}
  \Val(E) = \bigoplus_{i=0} ^n \Val_i (E).
\end{equation*}
\end{thrm}

Furthermore, every valuation $\phi$ can be decomposed uniquely into even and odd parts
\begin{equation*}
  \phi= \phi^{\even} + \phi^{\odd},
\end{equation*}
where $\phi^{\even} (-K) = \phi^{\even} (K), \phi^{\odd} (-K) = -\phi^{\odd} (K)$ for every $K \in \mathcal{K}(E)$. Thus we have the following decomposition
\begin{equation*}
  \Val(E) = \bigoplus_{i=0,...,n; \epsilon \in \{+, -\}} \Val_i ^{\epsilon} (E),
\end{equation*}
where $\Val_i ^{+}$ is the space of even valuations, and $\Val_i ^{-}$ is the space of odd valuations.

\subsubsection{Examples}
Let us present some examples of such valuations:
\begin{enumerate}
  \item The Euler characteristic $\chi$ which satisfies $\chi(K)=1$ for every $K\in \mathcal{K}(E)$ is a constant valuation.
  \item The Lebesgue measure $\vol(\cdot)$ belongs to $\Val_n (E)$.
  \item For any convex body $A$, the function $\phi: \mathcal{K}(E)\rightarrow \mathbb{R}$ defined by $\phi(K) = \vol(K+A)$ is in $\Val(E)$.
  \item Let $K_1, ...,K_r \in \mathcal{K}(E)$ be convex bodies, then there is a polynomial relation
     \begin{equation*}
     \vol(t_1K_1 +...+t_r K_r) = \sum_{i_1 +...+i_r =n} \frac{n!}{i_1 !i_2 !...i_r !} V(K_1 [i_1],...,K_r [i_r]) t_1 ^{i_1}...t_r ^{i_r},
     \end{equation*}
  where $t_i \geq 0$ and $K_j [i_j]$ denotes $i_j$ copies of $K_j$ and where the coefficient $V(K_1 [i_1],...,K_r [i_r])$ denotes the \emph{mixed volume}. Fix $A_1, ..., A_{n-k} \in \mathcal{K}(E)$, then the function $\psi: \mathcal{K}(E) \rightarrow \mathbb{R}$ defined by
  \begin{equation*}
    \psi(K):=V(K[k], A_1, ..., A_{n-k})
  \end{equation*}
  belongs to $\Val_k (E)$.
\end{enumerate}

\subsection{Alesker's irreducibility theorem}
The group $\GL(E)$ acts on $\Val(E)$ by
\begin{equation*}
 (g\cdot \phi) (K)= \phi(g^{-1} K).
\end{equation*}
Note that $\Val_i ^{{\even}}$ (resp. $\Val_i ^{\odd}$) is invariant under this action.

\begin{exmple}\label{exmple action}
Assume that $\phi\in \Val_i (E)$ is given by $\phi_{L_1,...,L_{n-i}}(K):=V(K[i], L_1, ...,L_{n-i})$, then
\begin{align*}
  (g\cdot \phi_{L_1,...,L_{n-i}})(K)&=V(g^{-1}(K)[i], L_1, ...,L_{n-i})\\
  &=|\det g|^{-1} V(K[i], g(L_1),...,g(L_{n-i}))\\
  &=|\det g|^{-1} \phi_{g(L_1),...,g(L_{n-i})} (K),
\end{align*}
which implies $g\cdot \phi_{L_1,...,L_{n-i}} = |\det g|^{-1} \phi_{g(L_1),...,g(L_{n-i})}$. In particular, if $|\det g| =1$, then  $g\cdot \phi_{L_1,...,L_{n-i}} = \phi_{g(L_1),...,g(L_{n-i})}$.

In the case of a general Radon measure $\mu$ on $\mathcal{K}(E)^{n-i}$ such that:
\begin{equation*}
\int_{\mathcal{K}(E)^{n-i}} V(\mathbf{B}[i], g(L_1),...,g(L_{n-i}))d\mu(L_1,...,L_{n-i}) < + \infty,
\end{equation*}
 we have
\begin{equation*}
  g\cdot \phi_\mu (K) = \frac{1}{|\det g|} \int_{\mathcal{K}(E)^{n-i}} V(K[i], g(L_1),...,g(L_{n-i}))d\mu(L_1,...,L_{n-i}).
\end{equation*}
In particular, if we set $g\cdot \mu(L_1,...,L_{n-i}) = \mu(g^{-1}(L_1),...,g^{-1}(L_{n-i}))$, then $g\cdot \phi_\mu = \frac{1}{|\det g|} \phi_{g\cdot\mu}$.
\end{exmple}

Alesker's irreducibility theorem \cite{aleskermcmullenconj} is one of the milestones of the modern development of valuation theory, it can be stated as follows:

\begin{thrm} [Alesker's irreducibility theorem]
 \label{thm_alesker_irreducibility}
As a $\GL(E)$-module, the natural representation of $\GL(E)$ on the space $\Val_i ^{+} (E)$ and $\Val_i ^{-} (E)$ is irreducible for every $i=0, 1, ..., n$ (that is, there is no proper closed $\GL(E)$-invariant subspace).
\end{thrm}

As an immediate consequence, the above irreducibility result implies McMullen's conjecture on mixed volumes: the valuations of the form $\phi(K) = \vol(K+A)$ span a dense subspace in $\Val(E)$; the mixed volumes span a dense subspace in $\Val(E)$. Moreover, the above theorem also implies in the same way that the linear combinations of valuations of the form $\phi(K)=V(K[i], \Delta[n-i])$, where $\Delta$ is a simplex in $E$, are dense in the space $\Val_i (E)$. This result will allow us to further introduce the cone spanned by mixed volume in  $\Val_i (E)$ for which we exhibit its properties (see Section \ref{sec positive}).

\subsection{Convolution and product of smooth valuations} \label{section_convolution}

\begin{defn}[Alesker]
A valuation $\phi\in \Val(E)$ on convex sets is called \emph{smooth} if the map \begin{equation*}
\GL(E)\rightarrow \Val(E),\ g\mapsto g\cdot \phi
\end{equation*}
is smooth as a map from a Lie group to a Banach space.
\end{defn}
As a smooth valuation $\phi$ induces a map $\GL(E) \to \Val(E)$ given by $g \mapsto g \cdot \phi \in \Val(E) $. The space of smooth valuations can be endowed with the topology of $\mathcal{C}^\infty$ functions on $\GL(E)$ with values in the Banach space $\Val(E)$ (this is usually called the Garding topology).
This topology is naturally stronger than the topology induced from $\Val(E)$, since there is a continuous injection
 \begin{equation*}
 \Val^\infty(E)   \hookrightarrow \Val(E).
 \end{equation*}
The space of smooth valuations is denoted by $\Val^\infty(E)$, it is dense in $\Val(E)$ and the representation of $\GL(E)$ in $\Val^\infty(E)$ is continuous (see e.g. \cite{AleskerFuBook}).

By McMullen's grading decomposition, we have
\begin{equation*}
   \Val^\infty(E) = \bigoplus_{i=0,...,n; \epsilon \in \{{+}, {-}\}} \Val_i ^{\epsilon, \infty} (E).
\end{equation*}

\begin{exmple}
Assume that $A_1,...,A_{n-i}\in \mathcal{K}(E)$ are strictly convex bodies with smooth boundary, then $\phi_{A_1,...,A_{n-i}}(-)=V(- [i]; A_1,...,A_{n-i})$ is in $\Val_i ^\infty (E)$.
\end{exmple}

\begin{exmple}
[$G$-invariant valuations]
Let $G \subset \SO(E)$ be a compact subgroup. Let $\Val^G (E)$ be the subspace of $\Val(E)$ of $G$-invariant valuations on convex sets. By \cite[Proposition 2.6, 2.7]{aleskerSurvey} (see also \cite{aleskerProduct}), the space $\Val^G (E)$ is finite dimensional if and only if $G$ acts transitively on the unit sphere of $E$, and under the assumption that $G$ acts transitively on the unit sphere of $E$ one has $\Val^G (E) \subset \Val^\infty (E)$.
\end{exmple}

An crucial ingredient in recent development of valuation theory (or algebraic integral geometry) is the product structure introduced by Alesker \cite{aleskerProduct} which relied heavily on the irreducibility theorem.

\begin{defn}[Product]
There exists a bilinear map
\begin{equation*}
  \Val ^\infty (E) \times \Val ^\infty (E) \rightarrow \Val ^\infty (E)
\end{equation*}
which is uniquely characterized by the following two properties:
\begin{enumerate}
  \item continuity;
  \item if $A, B\in \mathcal{K}(E)$ are strictly convex bodies with smooth boundary, then the product of $\phi_A (\cdot)=\vol(\cdot + A), \phi_B (\cdot)=\vol(\cdot + B)$ is given by
      \begin{equation*}
        \phi_A \cdot \phi_B (K) = \vol_{V\times V} (\Delta(K)+(A\times B)),
      \end{equation*}
      where $\Delta: E \rightarrow E\times E$ is the diagonal embedding.
\end{enumerate}
The product makes $\Val ^\infty (E)$ a commutative associative algebra with the unit given by the Euler characteristic.
\end{defn}

\begin{exmple}
(see \cite[Proposition 2.2]{aleskerProduct})
Assume that $A_1,...,A_{n-k}$ and $B_1, ..., B_{k}$ are strictly convex bodies with smooth boundary, then
\begin{equation*}
  V(-;A_1,...,A_{n-k})\cdot V(-;B_1, ..., B_{k})=\frac{k!(n-k)!}{n!}V(A_1,...,A_{n-k},-B_1, ..., -B_{k})\vol(-).
\end{equation*}
\end{exmple}

As a dual operation, Bernig and Fu introduced in \cite{bernigConvolution} another operation on smooth valuations refered as the convolution.
For simplicity, we fix a Euclidean product on $E$, then the main properties of the convolution can be summarized as follows.

\begin{defn}[Convolution]
There exists a bilinear map
\begin{equation*}
  \Val ^\infty (E) \times \Val ^\infty (E) \rightarrow \Val ^\infty (E)
\end{equation*}
which is uniquely characterized by the following two properties:
\begin{enumerate}
  \item the convolution is continuous;
  \item if $A, B\in \mathcal{K}(E)$ are strictly convex bodies with smooth boundary, then the convolution of $\phi_A (\cdot)=\vol(\cdot + A), \phi_B (\cdot)=\vol(\cdot + B)$ is given by
      \begin{equation*}
        \phi_A * \phi_B (K) = \vol (K+(A+B)).
      \end{equation*}
\end{enumerate}
The convolution makes $\Val ^\infty (E)$ a commutative associative algebra with the unit given by the volume.
\end{defn}

The following formula for the convolution of two valuations (see \cite[Corollary 1.3]{bernigConvolution}) will play an important role to extend this operation to arbitrary mixed volumes (see Section \ref{section extension *}).

\begin{prop}\label{prop convolution}
Assume that $A_1,...,A_{n-k}$ and $B_1, ..., B_{n-l}$ are strictly convex bodies with smooth boundary, and $k+l \geq n$, then
\begin{equation*}
  V(-;A_1,...,A_{n-k})*V(-;B_1, ..., B_{n-l})=\frac{k!l!}{n!}V(-;A_1,...,A_{n-k},B_1, ..., B_{n-l}).
\end{equation*}
\end{prop}

The product and convolution of smooth valuations are dual to each other by Alesker's Fourier transform.

\begin{thrm}[see \cite{alesker_fourier_type_transform}]
There is an algebra isomorphism given by \   $\widehat{}: (\Val^\infty(E), \cdot) \rightarrow (\Val^\infty (E), *)$ such that
\begin{equation*}
  \widehat{\phi\cdot \psi}=\widehat{\phi}*\widehat{\psi},\ \phi, \psi\in \Val^\infty (E),
\end{equation*}
for any smooth valuation $\phi, \psi$ on $E$.
\end{thrm}


\begin{rmk}
Comparing with the intersection theory in algebraic geometry, it is convenient to view $\Val_i ^\infty (E)$ as the analog of the group of numerical cycle classes of dimension $i$, the convolution  as the analog of the cup product of cohomology classes, the product  as the analog intersection of cycles and the Fourier transform  as the analog of the Poincar\'{e} duality. In this paper, we shall only consider the convolution operation on valuations rather than product operation.
\end{rmk}

\section{The extension of the convolution operator}\label{sec positive}

In this section, we consider a cone in the space of valuations on convex sets. Using its properties, we define two appropriate norms and the completions $\mathcal{V}^\mathcal{P}$ and $\mathcal{V}^\mathcal{C}$ of the space of valuations with respect to each of these norms. Finally, we prove that the convolution extends to a bilinear map on $\mathcal{V}^\mathcal{C}\times \mathcal{V}^\mathcal{P}$.

\subsection{The cone of $\mathcal{P}$-positive valuations}
By Alesker's irreducibility theorem, we know that the mixed volumes span a dense subspace in $\Val(E)$.
Let $\phi \in \Val_i (E)$, then for any $\varepsilon>0$ there exist valuations given by mixed volumes and real numbers $c_1, ..., c_m$ such that
\begin{equation*}
  ||\phi - \sum_{k=1} ^m c_k \psi_{k}|| \leq \varepsilon,
\end{equation*}
where $\psi_{k} (-) = V(-; K_1 ^k, ...,K_{n-i} ^k)\in \Val_i (E)$ for some  $K_1 ^k, ...,K_{n-i} ^k \in \mathcal{K}(E)$. This motivates the following definition for our cone.

For any positive Radon measure $\mu$ on $\mathcal{K}(E)^{n-i}$ such that
\begin{equation*}
\int_{\mathcal{K}(E)^{n-i}} V( \mathbf{B} [i], K_1,\ldots, K_{n-i}) d\mu(K_1, \ldots , K_{n-i}) < + \infty,
\end{equation*}
Denote by $\phi_\mu$ the map from $\mathcal{K}(E)$ to $\mathbb{R}$ given by:

\begin{equation*}
\phi_\mu (L) = \int_{\mathcal{K}(E)^{n-i}} V( L [i], K_1,\ldots, K_{n-i}) d\mu(K_1, \ldots , K_{n-i}),
\end{equation*}
where  $L \in \mathcal{K}(E)$ is a convex body.
We will see that for any Radon measure $\mu$ as above, the map $\phi_\mu$ defines a continuous translation invariant valuation (see Lemma \ref{lem_radon_valuation_def}).

\begin{defn}
We define the convex cone $\mathcal{P}_i \subset \Val_i(E)$ given by:
\begin{equation*}
\mathcal{P}_i := \left \{\phi_\mu| \phi_\mu (L) := \int_{\mathcal{K}(E)^{n-i}} V( L [i], K_1,\ldots, K_{n-i}) d\mu(K_1, \ldots , K_{n-i})  \right  \},
\end{equation*}
where $\mu$ is taken over the positive Radon measures on $\mathcal{K}(E)^{n-i}$ satisfying the following condition
\begin{equation*}
\int_{\mathcal{K}(E)^{n-i}} V(\mathbf{B} [i], K_1,\ldots, K_{n-i}) d\mu(K_1, \ldots , K_{n-i}) < + \infty.
\end{equation*}
We call a valuation $\phi\in\Val_i (E)$ on convex sets \emph{$\mathcal{P}$-positive} if $\phi$ belongs to the cone $ \mathcal{P}_i$.
\end{defn}

By definition, the cone $\mathcal{P}_i$ is convex.
Using a polarization argument, we observe that a valuation $\phi \in \mathcal{P}_i$ defines a unique function on $\mathcal{K}(E)^i$ defined by:
\begin{equation*}
\phi(L_1, \ldots, L_i) = \frac{1}{i!} {\dfrac{\partial^i}{\partial t_1 \partial t_2 \ldots \partial t_i}} \left ( \phi( t_1 L_1 + \ldots + t_i L_i) \right )_{| t_1 = \ldots = t_i = 0^+} ,
\end{equation*}
where $L_1, \ldots, L_i$ are convex bodies. In particular, $\phi(L, \ldots, L)=\phi(L)$.

\begin{defn}
We say that a valuation $\phi \in \mathcal{P}_i$ is \emph{strictly $\mathcal{P}$-positive} if there exists $\varepsilon>0$ such that:
\begin{equation*}
\phi(L_1, ..., L_i) \geqslant \varepsilon V(\mathbf{B} [n-i],L_1, ..., L_i)
\end{equation*}
for any convex body $L_1, ..., L_i \in \mathcal{K}(E)$.
\end{defn}

\begin{rmk}
The definition for ``strict positivity'' is inspired by the study of  positivity properties of cohomology classes in complex geometry. The convex body $\mathbf{B}$ can be viewed as a K\"ahler class, and the inequality defining strict positivity of $\phi_\mu$ can be viewed as the ``pseudo-effectivity'' of $\phi_\mu - \varepsilon V(\mathbf{B} [n-i];-)$ in some sense.
\end{rmk}

The following lemma proves that the cone $\mathcal{P}_i$ is a subset of $\Val_i(E)$.

\begin{lem} \label{lem_radon_valuation_def}
For any Radon measure $\mu $ on $\mathcal{K}(E)^{n-i}$ such that
\begin{equation*}
\int_{\mathcal{K}(E)^{n-i}} V(\mathbf{B} [i], K_1,\ldots, K_{n-i}) d\mu(K_1, \ldots , K_{n-i}) < + \infty,
\end{equation*}
the valuation
$\phi_\mu$ defines a continuous and translation invariant valuation.
\end{lem}

\begin{proof}
Let us first prove that the integral is well-defined. Take a convex body $L \in \mathcal{K}(E)$, there exists a constant $\lambda>0$ such that $L \subset \lambda \mathbf{B}$.
Since the mixed volume is monotone, we have:
\begin{align*}
\phi_\mu(L) &= \int_{\mathcal{K}(E)^{n-i}} V(L[i], K_1, \ldots, K_{n-i} ) d\mu(K_1, \ldots , K_{n-i}) \\
&\leq \lambda^i \int_{\mathcal{K}(E)^{n-i}} V(\mathbf{B}[i], K_1, \ldots, K_{n-i} ) d\mu(K_1, \ldots , K_{n-i})< + \infty.
\end{align*}

As $V(-; K_1, \ldots, K_{n-i})$ is a translation invariant valuation for any $K_1, \ldots , K_{n-i} \in \mathcal{K}(E)$, it is clear that $\phi_\mu $ is also a translation invariant valuation.
Let us prove that $\phi_\mu$ is continuous. Assume that $d_H (L_k, L) \rightarrow 0$, we need to check that $\phi_\mu (L_k) \rightarrow \phi_\mu (L)$. This is a direct consequence of the dominated convergence theorem.
\end{proof}

\begin{defn}
We denote by $\mathcal{V}_i'$ the subspace generated by $\mathcal{P}_i$, i.e., $\mathcal{V}_i': = \mathcal{P}_i - \mathcal{P}_i$.
\end{defn}

It follows directly from Alesker's irreducibility theorem that the space $\mathcal{V}_i'$ is a dense subspace of $\Val_i(E)$ (endowed with the norm $||\cdot ||$).

\begin{exmple}
When $\mu$ is a finite linear combination of Dirac measures on $\mathcal{K}(E)^{n-i}$, then the associated valuation $\phi_\mu \in \mathcal{V}_i'$ is a linear combination of mixed volumes.
\end{exmple}

\begin{exmple}\label{sec exmple}

Let us consider the positive cones $\mathcal{P}_1$ and $\mathcal{P}_{n-1}$:
\begin{enumerate}
  \item By Minkowski's existence theorem (see \cite{schneider_convex}), if $\mu$ is a positive Borel measure on $\mathbb{S}^{n-1}$ which is not concentrated on any great subsphere and has the origin as its center of mass, then $\mu$ is given by the surface area measure of a convex body with non-empty interior.
  In particular, for any $n-1$ convex bodies $K_1,..., K_{n-1}$ with non-empty interior, up to a translation, there is a unique convex body $K$ with non-empty interior such that
      \begin{equation*}
        V(-;K_1,..., K_{n-1})=V(-;K[n-1]).
      \end{equation*}
  By Minkowski's existence theorem again, for any two convex bodies $K, L$, up to a translation, there exists a unique convex body $M$ such that
      \begin{equation*}
       V(-; K[n-1])+V(-;L[n-1])=V(-;M[n-1]).
      \end{equation*}
  In particular, this proves that the set of strictly $\mathcal{P}$-positive elements in $\mathcal{P}_1$ is equal to
        \begin{equation*}
        \{V(-; K [n-1])| \ K \in \mathcal{K}(E)\ \textrm{with non-empty interior}\}.
        \end{equation*}
  Thus there exists a continuous linear map from  $\mathcal{P}_1$ to a convex cone in the space of Borel measures on $\mathbb{S}^{n-1}$. Indeed, consider  a valuation $\phi_\mu \in \mathcal{P}_1$,  we show that it gives a bounded linear functional on $C^0 (\mathbb{S}^{n-1})$ endowed with the topology of uniform convergence. For any $f \in C^0 (\mathbb{S}^{n-1})$, we have
  \begin{align*}
    &\phi_\mu (f):=\int_{\mathcal{K}(E)^{n-1}} d\mu(A_1,...,A_{n-1})\int_{\mathbb{S}^{n-1}} f d S(A_1,...,A_{n-1})\\
    &\leq |f|_\infty \int_{\mathcal{K}(E)^{n-1}} d\mu(A_1,...,A_{n-1})\int_{\mathbb{S}^{n-1}} h_\mathbf{B}d S(A_1,...,A_{n-1})\\
    &=\phi_\mu (\mathbf{B})|f|_\infty,
  \end{align*}
  where $d\mu(A_1,...,A_{n-1})$ is the surface area measure associated to $A_1,...,A_{n-1}$ and $h_\mathbf{B}$ is the support function of the unit ball which is equal to 1 on $\mathbb{S}^{n-1}$. Furthermore, if $\phi_\mu$ is strictly $\mathcal{P}$-positive, then by Minkowski's existence theorem there is a unique (up to a translation) convex body $K_\mu$ with non-empty interior such that
  $\phi_\mu = V(-; K_\mu [n-1])$.

  \item To further study the cone $\mathcal{P}_{n-1}$, we will apply the arguments presented in the proof of Theorem \ref{thrm_density_smooth} and Theorem \ref{thrm pos invariant}. In particular,  we will see that
        \begin{equation*}
        \mathcal{P}_{n-1} = \{V(-; K)| \ K \in \mathcal{K}(E)\}.
        \end{equation*}
       Since the space of convex bodies can be realized as a subspace of continuous functions on the sphere, we conclude that there exists a continuous linear map between $\mathcal{P}_{n-1}$ and a convex cone in the space of continuous function on the sphere.
\end{enumerate}

\end{exmple}

\begin{rmk}\label{rmk mcmullen charc}
The space $\Val_{n-1} (E)$ of valuations of degree $n-1$ was characterized by McMullen (see \cite{mcmullencharact}). Let $L(\mathbb{S}^{n-1})$ denote the space of the restriction of linear functions to the unit sphere, then there is an isomorphism between the quotient space $C^0 (\mathbb{S}^{n-1})/L(\mathbb{S}^{n-1})$ and $\Val_{n-1} (E)$. Thus for every $\phi\in \Val_{n-1} (E)$, up to a linear function, there is a unique continuous function $f_\phi$ such that for any $K\in \mathcal{K}(E)$, one has
\begin{equation*}
  \phi(K)=\int_{\mathbb{S}^{n-1}} f_\phi (x) dS(K^{n-1};x),
\end{equation*}
where $dS(K^{n-1};x)$ is the surface area measure of $K$.
By the correspondences established in \cite{lehmann2016correspondences}, the analog of the space $\Val_{n-1} (E)$ in complex geometry is the vector space of real numerical divisor classes on a projective variety, and the cone  $\mathcal{P}_{n-1}$ corresponds to the cone spanned by  movable numerical divisor classes.
As for the cone $\mathcal{P}_1$, we view the analog of its closure in complex geometry as the cone spanned by movable curve classes.
%

In higher degree, there is also a characterization of the general space $\Val_i (E) = \Val_i ^+ (E)\bigoplus \Val_i ^- (E)$  due to Klain-Schneider (see e.g. \cite[Section 2]{aleskermcmullenconj}, \cite[Section 2.4]{alesker_fourier_type_transform}). The space $\Val_i ^+ (E)$ can be $\GL(E)$-equivalently realized as a subspace of the space of smooth sections of certain line bundle over the Grassmannian $\Gr_i (E)$, and the space $\Val_i ^- (E)$ can be $\GL(E)$-equivalently realized as a subspace of the quotient of the space of smooth sections of certain line bundle over the partial flag space $\mathcal{F}_{i, i+1} (E)$. It would be thus interesting to understand the image of the cone $\mathcal{P}_i$ in these functional spaces.
\end{rmk}


\subsection{Mass transport estimates}

Consider two Radon measures $\mu, \nu$ on $\mathcal{K}(E)^{n-i}$ and $\mathcal{K}(E)^{n-j}$ respectively. Let $\phi_\mu \in \mathcal{V}'_i, \phi_\nu \in \mathcal{V}'_j$ be their associated valuations. We define the valuation $\phi_\mu* \phi_\nu $ given by:
\begin{align}\label{eq *}
\phi_\mu * \phi_\nu (-)
= \dfrac{i! j!}{n!} \int_{\mathcal{K}(E)^{2n-i-j}} V(-; A_1, \ldots , A_{n-i}, B_1 ,\ldots B_{n-j} ) d\mu(A) d\nu(B).
\end{align}
where $d\mu(A):=d\mu(A_1, \ldots , A_{n-i}), d\nu(B):=d\nu(B_1 , \ldots , B_{n-j} )$. We will see immediately that the integral in (\ref{eq *}) is well defined, that is, for any $D\in \mathcal{K}(E)$, $\phi_\mu * \phi_\nu (D)$ is finite (see Corollary \ref{cor * well defined}) and that it does not depend on the choice of the measures but on their associated valuations.
\medskip

The main ingredient to justify the above assertions are the following estimates taken from \cite[Theorem 5.9]{lehmann2016correspondences}.

\begin{thrm}\label{thrm rev KT}
Let $\phi\in \mathcal{P}_k$ and $\psi \in \mathcal{P}_{n-k}$, then for any $K \in \mathcal{K}(E)$ we have
\begin{equation*}
 \phi(K) \psi(K) \geq \vol(K)\phi*\psi.
\end{equation*}
\end{thrm}

\begin{proof} By definition, there exists two Radon measures $\mu$ and $\nu$ on $\mathcal{K}(E)^{n-k}$ and $\mathcal{K}(E)^{k}$ such that $\phi = \phi_\mu$ and $\psi = \phi_\nu$ respectively.
By definition, $\phi_\mu * \phi_\nu$ is equal to
\begin{equation*}
  \phi_\mu * \phi_\nu = \frac{k!(n-k)!}{n!}\int_{\mathcal{K}(E)^n}V(A_1,...,A_{n-k}, B_1,...,B_{k}) d\mu(A_1, ...,A_{n-k})d\mu(B_1, ...,B_{k}).
\end{equation*}

\textbf{Claim}: there is a constant $c>0$ depending only on $n, k$ such that
\begin{equation*}
  V(K[k];A_1,...,A_{n-k})
  V(K[n-k];B_1,...,B_{k}) \geq c V(A_1,...,A_{n-k},B_1,...,B_{k}) \vol(K).
\end{equation*}

The above inequality is just a slight generalization of
\cite[Theorem 5.9]{lehmann2016correspondences},
and the proof is similar. We refer to
\cite[Section 5]{lehmann2016correspondences} for the details (see also \cite{xiao2017bezout}). Let us give a sketch of the argument here.
Without loss of generality, we can assume the $A_l , B_l $ and $K$ are open and have non-empty interior. We apply a result of \cite{gromov1990convex}
and results from mass transport (see \cite{brenier1991polar, mccann95existence}). Then after solving a real Monge-Amp\`{e}re equation related to $K$, the desired geometric inequality of convex bodies can be reduced to an inequality for mixed discriminants -- the mixed discriminants given by the Hessian of those convex functions defining the convex bodies. More precisely, as in \cite{lehmann2016correspondences} (see also \cite{milman99masstransport}) the inequality for mixed volumes is reduced to an inequality for integrals:
\begin{align*}
  &\int_{\mathbb{R}^n} D(\nabla^2 f_{A_1}, ..., \nabla^2 f_{A_{n-k}}, (\nabla^2 F_{K})[k]) dx \int_{\mathbb{R}^n} D((\nabla^2 F_{K})[n-k], \nabla^2 f_{B_1}, ..., \nabla^2 f_{B_{k}}) dx\\
  &\geq \frac{k!(n-k)!}{n!} \int_{\mathbb{R}^n} \det (\nabla^2 F_{K})dx \int_{\mathbb{R}^n} D(\nabla^2 f_{A_1}, ..., \nabla^2 f_{A_{n-k}}, \nabla^2 f_{B_1}, ..., \nabla^2 f_{B_{k}})dx,
\end{align*}
where $\nabla^2$ is the Hessian operator, $D(-)$ denotes mixed discriminants, and $f_{A_i}, f_{B_j}, F_K$ are convex functions obtained by the results in \cite{gromov1990convex} and
\cite{ brenier1991polar, mccann95existence}.

Let $M_K$, $M_1, \ldots M_{n-k}$, $M_1' , \ldots , M_k'$ be the associated positive symmetric matrices given by $\nabla^2 F_K$, $\nabla^2 f_{A_1}, \ldots \nabla^2 f_{A_{n-k}}$, $\nabla^2 f_{B_1}, \ldots , \nabla^2 f_{B_k}$ respectively.
After an application of the Cauchy-Schwarz inequality
\begin{equation*}
(\int |fg| dv)^2 \leq (\int |f|^2 dv)( \int |g|^2 dv )
\end{equation*}
to the left hand side of the above inequality for integrals, the pointwise inequality needed is:
\begin{align*}
  D(M_K[k];M_1,...,M_{n-k})
  D(M_K[n-k];M_1',...,M_{k}')
  \geq \frac{k!(n-k)!}{n!} D(M_1,...,M_{n-k},M_1',...,M_{k}') \det(M_K).
\end{align*}
The above inequality for positive matrices is equivalent to an inequality for positive $(1,1)$-forms by replacing the positive matrices by positive $(1,1)$ forms and the discriminants by wedge product of differential forms (see e.g. \cite[Section 2]{xiao2017bezout}). Assume that $M=[a_{i\bar j}]$ is a positive Hermitian matrix, then it determines a positive $(1,1)$ form on $\mathbb{C}^n$ given by:
\begin{equation*}
  M \mapsto \omega_M := \sqrt{-1}\sum_{i,j} a_{i\bar j} dz^i \wedge d\bar{z}^j.
\end{equation*}
By this correspondence, the pointwise inequality for discriminants is equivalent to
\begin{equation*}
  (\omega_{M_K} ^k \wedge \omega_{M_1}\wedge... \omega_{M_{n-k}})
  (\omega_{M_K} ^{n-k} \wedge \omega_{M_1'}\wedge...\wedge \omega_{M_{k}'}) \geq \frac{k!(n-k)!}{n!} \omega_{M_K} ^n (\omega_{M_1}\wedge... \omega_{M_{n-k}} \wedge \omega_{M_1'}\wedge...\wedge \omega_{M_{k}'}).
\end{equation*}

Note that wedge products of positive $(1,1)$ forms are Hermitian positive. More generally, assume that $\Phi$ is a Hermitian positive $(n-k, n-k)$ form, $\Psi$ is a Hermitian positive $(k,k)$ form and $\omega$ is a positive $(1,1)$ form\footnote{For the positivity of forms, we refer the reader to \cite[Chapter 3]{Dem_AGbook} and \cite[Section 1]{debarre2011pseudoeffective}.
In \cite[Definition 1.4]{debarre2011pseudoeffective}, ``Hermitian positive'' is called semipositive.}, then
\begin{equation}\label{eq form KT}
  (\Phi \wedge \omega^k)(\omega^{n-k}\wedge \Psi) \geq \frac{k!(n-k)!}{n!} (\Phi\wedge \Psi) \omega^n.
\end{equation}
Recall that a $(l, l)$ form is Hermitian positive on the vector space $\mathbb{C}^n$ if its associated Hermitian
form on $\wedge^l \mathbb{C}^n$ is semipositive (see \cite[Definition 1.4]{debarre2011pseudoeffective}), that is, the coefficients of the $(l,l)$ form give a semipositive Hermitian matrix on $\wedge^l \mathbb{C}^n$, here $\wedge^l \mathbb{C}^n$ is the $l$-th wedge product of $\mathbb{C}^n$.
By taking some local coordinates, it is sufficient to check the above inequality when $\omega$ is given by the identity matrix. As $\Phi$ is Hermitian positive, then
\begin{equation*}
 \sum_{|J|=n-k} (\sum_{|I|=n-k} \Phi_{I, I}) dz_J \wedge d\bar{z}_{J} - \Phi
\end{equation*}
is also Hermitian positive. As $\Psi$ is Hermitian positive and the cone generated by Hermitian positive $(k,k)$ forms is dual to the cone generated by Hermitian positive $(n-k,n-k)$ forms
(see
\cite[Section 1]{debarre2011pseudoeffective}), we get
\begin{equation*}
 \left(\sum_{|J|=n-k} (\sum_{|I|=n-k} \Phi_{I, I}) dz_J \wedge d\bar{z}_{J} - \Phi\right)\wedge \Psi \geq 0,
\end{equation*}
which gives the desired pointwise inequality (\ref{eq form KT}).

In summary, we finally obtain
\begin{equation*}
 \phi(K) \psi(K) \geq \vol(K)\phi*\psi,
\end{equation*}
as required.
\end{proof}

\begin{rmk}
The above inequality is refered in \cite{lehmxiao16convexity} as the ``reverse Khovanskii-Teissier inequality''. The main reason is that the classical Khovanskii-Teissier inequality provides a lower bound on $\phi *\psi$ whereas
 the above inequality gives an upper bound:
\begin{equation*}
  \phi*\psi \leq \inf_{\vol(K)=1}  \phi(K)\psi(K).
\end{equation*}
The previous estimate has a striking analog in complex geometry which was first obtained using Demailly's holomorphic Morse inequality
(see e.g. \cite{Siu93matsusaka}, \cite[Chapter 8]{dem_analyticAG}).
For further discussion on the relationship between these estimates in both settings, we shall refer to \cite{lehmann2016correspondences}.
\end{rmk}

\begin{rmk} Observe that the above estimate can be extended further as in \cite[Theorem 1.1]{xiao2017bezout}.
\end{rmk}

%
%

\bigskip



Let us illustrate how the previous estimates can be applied to prove that our operator $*$ is well-defined.

\begin{prop}\label{lem * well_defined}
The operator $*$ defined by the formula (\ref{eq *}) induces a bilinear map $*: \mathcal{V}_i' \times \mathcal{V}_j ' \rightarrow \mathcal{V}_{i+j-n} '$.
\end{prop}

\begin{proof}
This proposition follows immediately from the following Lemma \ref{cor * well defined} and Lemma \ref{lem * indp}.
\end{proof}

\begin{lem}\label{cor * well defined}
For any $\phi_\mu \in \mathcal{P}_i, \psi_\nu \in \mathcal{P}_j$, the integral (\ref{eq *}) defining $\phi_\mu * \psi_\nu$ is well defined.
\end{lem}

\begin{proof}
We only need to check that the integral defining $\phi_\mu * \psi_\nu (D)$ is well defined, when $D$ has non-empty interior. This follows directly from Theorem \ref{thrm rev KT}.
\end{proof}

Observe that that different Radon measures may give the same valuation, we prove however that $\phi_\mu * \phi_\nu$ depends only on the valuation and not on the particular choice of a Radon measure.

\begin{lem}\label{lem * indp}
The valuation $\phi_\mu * \phi_\nu$ does not depend on the choice of the measures $\mu$ and $ \nu$.
\end{lem}

\begin{proof}
Consider Radon measures $\mu_1, \mu_2$ on $\mathcal{K}(E)^{n-i}$ and  $\nu_1$, $\nu_2$ on $\mathcal{K}(E)^{n-j}$ respectively.
Assume that $\phi_{\mu_1}=\phi_{\mu_2}, \phi_{\nu_1}=\psi_{\nu_2}$, we prove that $\phi_{\mu_1} * \phi_{\nu_1} = \phi_{\mu_2} * \phi_{\nu_2}$.

We need to verify that for any $L\in \mathcal{K}(E)$,
\begin{align*}
&\int_{\mathcal{K}(E)^{2n-i-j}} V(L[i+j-n]; A_1, \ldots , A_{n-i}, B_1 ,\ldots B_{n-j} ) d\mu_1 (A) d\nu_1 (B)\\
&= \int_{\mathcal{K}(E)^{2n-i-j}} V(L[i+j-n]; A_1, \ldots , A_{n-i}, B_1 ,\ldots B_{n-j} ) d\mu_2 (A) d\nu_2 (B).
\end{align*}

For any $t= (t_1, \ldots , t_j) \in (\mathbb{R}^+)^j$, denote by $K_t=t_1 K_1 +...+t_j K_j$  where $K_1,...,K_j $ are convex bodies.
Since $\phi_{\nu_1}=\phi_{\nu_2}$, we have that $\phi_{\nu_1}(K_t)=\phi_{\nu_2}(K_t)$.
Since $\phi_{\nu_i}(K_t)$ is a polynomial in $t_1,...,t_j$, the equality on the coefficients of the polynomial gives
\begin{equation*}
  \int_{\mathcal{K}(E)^{n-j}} V(K_1,...,K_j; B_1 ,\ldots B_{n-j} )d\nu_1 (B) = \int_{\mathcal{K}(E)^{n-j}} V(K_1,...,K_j; B_1 ,\ldots B_{n-j} )d\nu_2 (B).
\end{equation*}
In particular, this implies that $\phi_{\mu_1}*\phi_{\nu_1} = \phi_{\mu_1}*\phi_{\nu_2}$. Similarly, we also obtain that $\phi_{\mu_1}*\phi_{\nu_2} = \phi_{\mu_2}*\phi_{\nu_2}$, hence $\phi_{\mu_1}*\phi_{\nu_1} = \phi_{\mu_2}*\phi_{\nu_2}$, as required.
\end{proof}

\subsection{The norm $|| \cdot ||_\mathcal{P}$ and the completion $\mathcal{V}^\mathcal{P}$}
In this section, we give the definition of norm $|| \cdot ||_\mathcal{P}$, and consider the completion denoted $\mathcal{V}_i^\mathcal{P}$ of $\mathcal{V}_i'$ with respect to this norm.
We then show that the subspace $\mathcal{P}_i \cap \Val^\infty(E)$ of smooth valuations is dense in $\mathcal{V}_i^\mathcal{P}$.

\begin{defn}
For any valuation $\phi \in \mathcal{V}'_i$, we define $||\phi ||_{\mathcal{P}}$ by the following formula.
\begin{equation*}
  ||\phi ||_{\mathcal{P}} := \inf\{t\geq 0|\ | \phi(L_1,...,L_i)| \leq t V(\mathbf{B}[n-i], L_1, ...,L_i)\ \textrm{for any}\ L_1, ...,L_i \in \mathcal{K}(E)\}.
\end{equation*}
\end{defn}

First we show that for any $\phi \in \mathcal{V}'_i$, $||\phi||_\mathcal{P}$ is well defined  and we denote by $\mathcal{V}_i^\mathcal{P}$ the completion of $\mathcal{V}_i'$ with respect to this norm. Let us also set $\mathcal{V}^\mathcal{P} := \oplus_{i=0}^n \mathcal{V}^\mathcal{P}$.

\begin{prop}\label{prop P norm welldefined}
The map $ || \cdot ||_{\mathcal{P}} : \mathcal{V}'_i \to \mathbb{R}^+$ defines a norm on $\mathcal{V}'_i$.
\end{prop}

\begin{proof}
The only fact which is not straightforward is whether $|| \cdot ||_{\mathcal{P}_i}$ is well-defined.
Consider $\phi \in \mathcal{V}'_i$, we prove that there exists a $t>0$ such that
\begin{equation*}
  |\phi(L_1,...,L_i)|\leq t V(\mathbf{B}[n-i], L_1, ...,L_i).
\end{equation*}
 By definition, there exists a signed Radon measure $\mu$ on $\mathcal{K}(E)^{n-i}$ such that $\phi = \phi_\mu$.
Consider the Hahn decomposition $\mu = \mu^+ - \mu^-$ of the measure $\mu$ so that  $\phi_\mu = \phi_{\mu^+} - \phi_{\mu^-}$.
One has that
\begin{equation*}
  |\phi(L_1,...,L_i)|\leq \phi_{\mu^+}(L_1,...,L_i)+\phi_{\mu^-}(L_1,...,L_i).
\end{equation*}
Let us find an upper bound for  $\phi_{\mu^+}(L_1,...,L_i)$. By Theorem \ref{thrm rev KT} we have
\begin{align*}
  \phi_{\mu^+}(L_1,...,L_i)& = \int_{\mathcal{K}(E)^{n-i}}V(L_1,...,L_i,K_1, ...,K_{n-i}) d\mu^+ (K)\\
  &\leq c V(\mathbf{B}[n-i], L_1,...,L_i) \int_{\mathcal{K}(E)^{n-i}}V(\mathbf{B}[i],K_1, ...,K_{n-i}) d\mu^+ (K),
\end{align*}
where $c>0$ depends only on $n, i, \vol(\mathbf{B})$. Since $\phi_{\mu^+}\in \mathcal{P}_i$, we get
\begin{equation*}
  \phi_{\mu^+}(L_1,...,L_i)\leq t V(\mathbf{B}[n-i], L_1, ...,L_i)
\end{equation*}
for some $t>0$. Similar estimates also hold for $\phi_{\mu^-}$, this proves that
$|| \phi ||_{\mathcal{P}} < + \infty$.

\end{proof}

\begin{rmk} \label{rmk p norm}
Observe that by homogeneity for $L_1, ...,L_i$, we have
\begin{equation*}
  ||\phi ||_\mathcal{P} := \inf\{t\geq 0|\ |\phi(L_1,...,L_i)|\leq t V(\mathbf{B}[n-i], L_1, ...,L_i)\ \textrm{for any}\ L_1, ...,L_i \subset \mathbf{B}\}.
\end{equation*}

\end{rmk}

By the above remark, we get:

\begin{prop}\label{prop norm ineq1}
For any $\phi \in \mathcal{V}' _i$, $||\phi|| \leq \vol(\mathbf{B})||\phi ||_\mathcal{P}$. Hence, there is a continuous injection:
\begin{equation*}
(\mathcal{V}_i^\mathcal{P} , || \cdot ||_{\mathcal{P}} ) \hookrightarrow (\Val_i(E) , || \cdot || ).
\end{equation*}
\end{prop}

Next we show that the smooth valuations are dense in $\mathcal{V}^\mathcal{P}$.

\begin{thrm} \label{thrm_density_smooth}
The space of finite sums of mixed volumes of convex bodies with strictly convex and smooth boundary is dense in $\mathcal{V}_i'$ for the topology induced by the norm $|| \cdot ||_{\mathcal{P}}$.
In particular, the space $\mathcal{V}_i' \cap \Val^\infty(E)$ is dense in $\mathcal{V}_i'$ for the topology induced by the norm $|| \cdot ||_{\mathcal{P}}$.
\end{thrm}

\begin{proof}  Since $\mathcal{V}_i'$ is generated by $\mathcal{P}_i$, we are reduced to prove the density of smooth valuations in $\mathcal{P}_i$. We prove that the finite sums of mixed volumes of convex bodies with strictly convex and smooth boundary are dense in $\mathcal{V}_i'$.

We prove it in two steps.

\medskip

\textbf{Step 1}: Let us first prove that the valuations in $ \mathcal{P}_i$ such that their associated measure has bounded support are dense in $\mathcal{P}_i$.

Take $\phi \in \mathcal{P}_i$ such that $\phi = \phi_\mu $ where $\mu$ is its associated positive Radon measure on $\mathcal{K}(E)^{n-i}$.
For any integer $k >0$, we consider the measure $\mu_k$ given by $\mu_k = \mu_{| B(0, k)}$, where
\begin{equation*}
B(0, k) = \left \{(K_1, \ldots , K_{n-i}) \in \mathcal{K}(E)^{n-i}\  |\ K_j \subset k \mathbf{B}, \forall\ 0 \leqslant  j \leqslant n-i    \right  \} \subset \mathcal{K}(E)^{n-i} .
\end{equation*}
By construction, the measure $\mu_k$ has bounded support. (By Blaschke selection theorem, $B(0, k)$ is a compact set.) By the monotone convergence theorem, we have that:
\begin{equation*}
\phi_{\mu_k}(L) = \int_{\mathcal{K}(E)^{n-i}} V(K_1, \ldots , K_{n-i} , L[i]) d\mu_k(K_1, \ldots , K_{n-i} ) \rightarrow \phi(L).
\end{equation*}

Let us prove that $|| \phi_\mu - \phi_{\mu_k} ||_{\mathcal{P}}$ converges to zero as $k \rightarrow + \infty$.
Fix some convex bodies $L_1, \ldots L_i$.
By construction, one has that:
\begin{equation*}
 0 \leqslant \phi(L_1, \ldots, L_i) - \phi_{\mu_k}(L_1, \ldots , L_i).
\end{equation*}
Moreover, by Theorem \ref{thrm rev KT} applied to $\phi' = V(K_1, \ldots, K_{n-i} , -[i])$ and $\psi' = V(L_1, \ldots , L_i , -[n-i] )$ and to the convex body $\mathbf{B}$, there exists a constant $C>0$ such that we have:
\begin{equation*}
V(K_1, \ldots, K_{n-i} , L_1 , \ldots ,L_i)  \leqslant C \dfrac{V(K_1, \ldots , K_{n-i}, \mathbf{B}[i])}{\vol(\mathbf{B})} V( \mathbf{B}[n-i], L_1, \ldots, L_i ) .
\end{equation*}
Integrating on the previous inequality, one obtains:
\begin{align*}
&\phi(L_1, \ldots, L_i) - \phi_{\mu_k}(L_1, \ldots , L_i)\\
&\leqslant \dfrac{C}{\vol(\mathbf{B})} \left ( \int_{ B(0, k )^c} V(K_1, \ldots , K_{n-i} , \mathbf{B}[i] )d\mu(K_1, \ldots, K_{n-i})  \right )
V(\mathbf{B}[n-i], L_1, \ldots , L_i  ) ,
\end{align*}
where $B(0,k)^c = \mathcal{K}(E)^{n-i} \setminus B(0,k)$.
We have thus proved:
\begin{equation*}
|(\phi- \phi_{\mu_k})(L_1, \ldots , L_i)| \leqslant \dfrac{C}{\vol(\mathbf{B})} (\phi(\mathbf{B}) - \phi_{\mu_k}(\mathbf{B})) V(\mathbf{B}[n-i], L_1, \ldots, L_i),
\end{equation*}
for any convex bodies $L_1, \ldots L_i$.
Since $\phi(\mathbf{B}) - \phi_{\mu_k}(\mathbf{B}) \rightarrow 0$ as $k \rightarrow + \infty$, we have that
$||\phi - \phi_{\mu_k}||_{\mathcal{P}} \rightarrow 0$ as required.

\medskip

\textbf{Step 2}: Suppose that $\phi = \phi_\mu \in \mathcal{P}_i$ is a valuation where  $\mu$ is a Radon measure on $\mathcal{K}(E)^{n-i}$ whose support is bounded. We prove that $\phi$ can be approached by $\phi_k$, where $\phi_k \in \mathcal{P}_i \cap \Val_i^\infty(E)$ is a finite sum of mixed volumes given by convex bodies with strictly convex and smooth boundary.

Suppose that the support of $\mu $ is contained in $B(0,  N)$ where $N >0$ is an integer.
For any $\epsilon>0$,  there exists a partition $\cup_{j=1} ^m O_j$ of $B(0,N)$ such that for any $(K_1, \ldots, K_{n-i} ) , (K'_1, \ldots , K_{n-i}') \in O_j $, one has:
\begin{equation}\label{eq d_H}
  d_H (K_j, K'_j) \leq \epsilon,\ \forall\ 1\leq j\leq n-i.
\end{equation}

Since the valuations given by mixed volumes are monotone and since $\supp \mu \subset B(0, N)$, there is a constant $C>0$ (depending only on $N, i$) such that
\begin{equation} \label{eq_thm_density_approx}
|V( K_1, \ldots ,K_{n-i} , L_1, \ldots, L_i ) - V( K_1', \ldots, K_{n-i}' , L_1, \ldots, L_i) | \leqslant C\epsilon V(\mathbf{B}[n-i], L_1, \ldots, L_i).
\end{equation}

Let us define the measure $\mu_{\epsilon}$ given by
\begin{equation*}
\mu_{\epsilon} := \sum_{j=1} ^m \mu(O_j) \delta_{(K_{1}^j, \ldots, K_{n-i}^j )},
\end{equation*}
where $(K_1^j , K_2^j , \ldots, K_{n-i}^j) \in O_j$ satisfying that $K_1^j ,\ldots K_{n-i}^j$ are convex bodies with smooth and strictly convex boundary, and where $\delta_{(K_1^j , K_2^j , \ldots, K_{n-i}^j)}$ is the dirac mass at the point $(K_1^j , K_2^j , \ldots, K_{n-i}^j)$.
Let us estimate the norm $|| \phi_{\mu_\epsilon} - \phi ||_{\mathcal{P}}$.
Take $L_1, \ldots, L_i \in \mathcal{K}(E)$.
By definition, one has that
\begin{align*}
\phi_{\mu_\epsilon} (L_1, \ldots, L_i ) &= \sum_{j=1} ^m \mu(O_j) V(K_1^j , \ldots, K_{n-i}^j, L_1, \ldots , L_i) ,\\
& =   \sum_{j=1} ^m \int_{O_j} V(K_1^j , \ldots, K_{n-i}^j, L_1, \ldots , L_i) d\mu(K_1, \ldots, K_i)  .
\end{align*}
The difference $|\phi_{\mu_\epsilon}(L_1, \ldots , L_i) - \phi(L_1, \ldots, L_i)| $ is bounded by:
\begin{align*}
&|\phi_{\mu_\epsilon}(L_1, \ldots, L_i) - \phi(L_1, \ldots, L_i)|\\
 &\leqslant
 \sum_{j=1}^m \left  |\int_{O_j} (V(K_1^j, \ldots, K_{n-i}^j , L_1, \ldots, L_i) - V(K_1, \ldots, K_{n-i} , L_1, \ldots, L_i)) d\mu(K_1, \ldots, K_{n-i} ) \right |\\
&\leqslant  \sum_{j=1}^m \int_{O_j} |V(K_1^j, \ldots, K_{n-i}^j , L_1, \ldots, L_i) - V( K_1, \ldots, K_{n-i} , L_1, \ldots, L_i)|d\mu(K_1, \ldots, K_{n-i}).
\end{align*}
Applying \eqref{eq_thm_density_approx} to the previous inequality, we obtain the following upper bound:
\begin{equation*}
|\phi_{\mu_\epsilon}(L_1, \ldots, L_i) - \phi(L_1, \ldots, L_i)| \leqslant C\epsilon \sum_{j=1}^m \int_{O_j} V( \mathbf{B}[n-i], L_1, \ldots, L_i ) d\mu(K_1, \ldots, K_{n-i}) .
\end{equation*}
Hence,
\begin{equation*}
|\phi_{\mu_\epsilon}(L_1, \ldots, L_i) - \phi(L_1, \ldots, L_i)| \leqslant C\epsilon V(\mathbf{B}[n-i], L_1, \ldots , L_i) \mu( B(0, N)) ,
\end{equation*}
and this implies that $||\phi_{\mu_\epsilon} - \phi ||_{\mathcal{P}} \leqslant C\epsilon \mu(B(0,N))$ is arbitrary small since $\mu(B(0,N))$ is finite.

\medskip

We have thus proven that finite sums of mixed volumes of convex bodies with smooth and strictly convex boundary are dense in $\mathcal{P}_i$ with respect to the norm $|| \cdot ||_\mathcal{P}$ as required.
\end{proof}

The above theorem has the following direct consequence.

\begin{cor}\label{cor density n-1}
The set of valuations $\left \{V(L;-[n-1])\ |\ L\in \mathcal{K}(E) \right \}$ is dense in $\mathcal{P}_{n-1}$ with respect to the topology given by $||\cdot||_\mathcal{P}$.
\end{cor}

We now study the relationship between the space $\mathcal{V}'$ with the space of smooth valuations.

Note that $\Val_i^{\infty}(E)$ has a decomposition into even valuations and odd valuations, i.e., $$\Val_i^{\infty}(E) = \Val_i^{\infty, +}(E) \bigoplus \Val_i^{\infty, -}(E). $$
So far, we obtained the following partial result.

\begin{prop}\label{smooth valuation inclusion}
For any integer $i \leqslant n$, we have the inclusion $\Val_i^{\infty, +}(E) \subset \mathcal{V}_i'$.
\end{prop}

\begin{proof}

Take an even valuation $\phi \in \Val_i^{\infty, +}(E)$.
By \cite[Proposition 6.1.5]{alesker_fourier_type_transform} (see also \cite{bernig2012algebraic}), there exists a smooth measure $dm$ on $ \Gr_i(E) $ such that

\begin{equation*}
\phi(K) = \int_{H \in \Gr_i(E)} \vol_H(\pi_H(K)) dm(H),
\end{equation*}
where $\pi_H : E \to H$ is the orthogonal projection onto $H$ and where $\vol_H$ denotes the volume on $H$.
By the projection formula (see \cite[Theorem 5.3.1]{schneider_convex}), take $B\in \mathcal{K}(E)$ with non-empty interior and let $B_{H} := B \cap H^\perp$, one has that:
\begin{equation*}
 \left ( \begin{array}{l}
n \\
i
\end{array} \right ) V(B_H [n-i]; K[i]) = \vol_{H^\perp}(B_H) \vol_H(\pi_H(K)),
\end{equation*}
where $\pi_H : E \to H$ is the orthogonal projection onto $H$.
In particular, one has that:
\begin{equation*}
\phi(K) = \int_{H \in \Gr_i(E)} \vol_{H^\perp}(B_H)^{-1} \left ( \begin{array}{l}
n \\
i
\end{array}\right )^{-1} V(B_H[n-i]; K[i]) dm(H).
\end{equation*}
Take $\mu$ to be the push-forward of the measure $\vol_{H^\perp}(B_H)^{-1} \left ( \begin{array}{l}
n \\
i
\end{array}\right )^{-1} dm(H)$ by the continuous map:
\begin{equation*}
H \in \Gr_i(E) \rightarrow B_H^{n-i} \in \mathcal{K}(E)^{n-i}.
\end{equation*}
Then we have that:
\begin{equation*}
\phi(K) = \int V(K[i]; L_1; \ldots ; L_{n-i}) d\mu(L_1, \ldots, L_{n-i}),
\end{equation*}
and $\phi\in \mathcal{V}_i'$ as required.

\end{proof}

We believe that the inclusion $\Val_i^{-,\infty} (E) \subset \mathcal{V}_i'$ also holds, but this second inclusion should be a consequence of the following fact.

\begin{conj} \label{claim1}
If $\varphi$ is a smooth valuation of degree $i$ on $\mathbb{R}^{i+1}$, then there exists a smooth function $f$ on the sphere $\mathbb{S}^i$  such that
\begin{equation*}
\varphi(K) = \int_{\mathbb{S}^i} f dS(K)
\end{equation*}
for all convex body $K$ in $\mathbb{R}^{i+1}$.
\end{conj}

\begin{rmk}
As communicated to us by Professor A.~Bernig, the above fact seems to be a standard corollary of representation theory, in particular, a consequence of Casselman-Wallach theorem (see e.g. \cite{casselman}, \cite{wallachVol1, wallachVolume2}).
\end{rmk}

\begin{prop}\label{odd val inclusion}
Assume that Conjecture \ref{claim1} holds, then the inclusion $\Val_i^{\infty , -} (E) \subset \mathcal{V}_i'$ is satisfied, hence $\Val_i^\infty(E) \subset \mathcal{V}_i'$.
\end{prop}

\begin{proof}
Let us prove that the inclusion $\Val_i^{\infty , -} (E) \subset \mathcal{V}_i'$ holds.
Take $\phi$ a smooth odd valuation of degree $i$. Again, using \cite[Proposition 6.1.5]{alesker_fourier_type_transform}, there exists a smooth measure $m$ on $\Gr_{i+1}(E)$ and a smooth function $\varphi : \Gr_{i+1}(E) \to \Val_{i}(E)$ such that for any $H \in \Gr_{i+1}(E)$, one has $\varphi(H)_{|H} \in \Val_{i}^{\infty,-}(H)$ and
\begin{equation*}
\phi(K) = \int_{H\in \Gr_{i+1}(E)} \varphi(H)(\pi_H(K)) dm(H),
\end{equation*}
where $dm$ is a smooth measure on $\Gr_{i+1}(E)$.

Since $\varphi(H)_{|H}$ is a smooth valuation of degree $i$ on $H$, there exists a continuous function $f_H$ on the sphere $\mathbb{S}^{i}$ such that for any convex body $K$ we have:
\begin{equation*}
\phi(K) = \int_{H\in \Gr_{i+1}(E)} \int_{\mathbb{S}^{i}} f_H(x) dS(x,\pi_H(K)) dm(H),
\end{equation*}
where $dS( x, \pi_H(K))$ denotes the area measure of the convex body $\pi_H(K)$.

Then for each point $H$ of the Grassmannian $\Gr_{i+1}(E)$, the function $f_H$ is a smooth function on the sphere.
We now state the following lemma (see \cite[Lemma 1.7.8]{schneider_convex}).
\begin{lem} \label{lem_decomp}Let $f$ be a smooth function on the sphere $\mathbb{S}^{i}$. Then there exists two convex bodies $L,L'$ (uniquely determined up to translation) in $\mathbb{R}^i$ and a constant $C>0 $ such that both $L$ and $L'$ are contained in the ball of radius $
 C \sup \{ |f|, |df| , |d^2 f| \} $ and
 such that $f = h_{L} - h_{L'}$.
\end{lem}

Using Lemma \ref{lem_decomp}, for any $H \in \Gr_{i+1}(E)$, there exists two convex bodies $L(H),L'(H)$ in $H$ contained in the ball of radius
\begin{equation*}
 C \sup \{ |f_H|, |df_H| , |d^2 f_H| \}
\end{equation*}
and such that $f_H = h_{L(H)} - h_{L'(H)}$.
In particular, this implies that
\begin{equation*}
\phi(K) = \int_{H\in \Gr_{i+1}(E)}(V(L(H); \pi_H(K) [i]) - V(L'(H); \pi_H(K) [i])) dm(H).
\end{equation*}
By the projection formula (see \cite[Theorem 5.3.1]{schneider_convex}), we have:
\begin{equation*}
\phi(K) = \int_{H\in \Gr_{i+1}(E)}C (V(L(H); K [i]; B \cap H^\perp [n-i-1]) - V(L'(H); K ; B \cap H^\perp [i])) dm(H),
\end{equation*}
where $C = (n-k)! k! / ((n!)\vol_{H^\perp} (B\cap H^\perp ) )$.
Consider the pushforward $\mu^+$ and $\mu^-$ of the measure $C dm$ by the continuous map
\begin{equation*}
 H \in \Gr_{i+1}(E) \mapsto (L(H), B\cap H^\perp [n-i-1]) \in \mathcal{K}(E)^{n-i}
\end{equation*}
and by the continuous map:
\begin{equation*}
 H \in \Gr_{i+1}(E) \mapsto (L'(H), B\cap H^\perp [n-i-1]) \in \mathcal{K}(E)^{n-i}
\end{equation*}
respectively.
By construction the signed measure $\mu$ defined by $\mu = \mu+ - \mu-$ satisfies:
\begin{equation*}
\phi(K) = \int_{\mathcal{K}(E)^{n-i}} V( L_1, \ldots , L_{n-i} , K[i]) d\mu(L_1, \ldots , L_{n-i}).
\end{equation*}
\end{proof}

\subsection{The norm $||\cdot||_{\mathcal{C}}$ and the completion $\mathcal{V}^\mathcal{C}$}
\label{sec cone norm}
Recall that the vector space $\mathcal{V}'_i$ is generated by $\mathcal{P}_i$, we shall define in this section a norm $||\cdot ||_\mathcal{C}$ and consider the completion with respect to this norm. This construction is inspired by the construction in algebraic geometry (see \cite{dang2017degrees}).

\begin{defn}
For any $\phi \in \mathcal{V}'_i$, we define $|| \phi ||_\mathcal{C}$ by the following formula:
\begin{equation*}
  ||\phi||_{\mathcal{C}}:=\inf_{\phi = \phi_{+} -\phi_{-}, \phi_{\pm}\in \mathcal{P}_i} \phi_{+}(\mathbf{B})+\phi_{-}(\mathbf{B}).
\end{equation*}

\end{defn}
Here, the symbol $\mathcal{C}$ stands for cone since this norm is induced by the cone $\mathcal{P}_i$ in $\mathcal{V}_i'$.

\begin{rmk}
Equivalently, by the Jordan decomposition of signed measures we have
$|| \phi_\mu ||_{\mathcal{C}} := \phi_{|\mu|}(\mathbf{B})$,
where $|\mu|$ is the absolute value of a Radon measure $\mu$ on $\mathcal{K}(E)^{n-i}$.
\end{rmk}

\begin{rmk} By construction, if $\phi \in \mathcal{P}_i$, then $|| \phi ||_{\mathcal{C}} = \phi(\mathbf{B}) $.
\end{rmk}

\begin{lem}
The function $||\cdot ||_{\mathcal{C}}$ defined above is a norm on the space $\mathcal{V}'_i $.
\end{lem}

\begin{proof}
It is clear that:
\begin{itemize}
  \item For any $c\in \mathbb{R}$ and any $\phi \in \mathcal{V}_i'$, we have $||c\phi||_{\mathcal{C}} = |c| ||\phi||_{\mathcal{C}}$;
  \item For any $\phi, \psi \in \mathcal{V}_i '$, we have $||\phi+ \psi||_{\mathcal{C}} \leq ||\phi||_{\mathcal{C}} +||\psi||_{\mathcal{C}}$.
\end{itemize}

It remains to verify:
\begin{itemize}
  \item If $||\phi||_{\mathcal{C}} =0$, then $\phi =0$.
\end{itemize}
To that end, take a sequence of decompositions $\phi = \phi_k ^+ - \phi_k ^-$ such that $\phi_k ^+ (\mathbf{B}) + \phi_k ^- (\mathbf{B})\rightarrow 0$.  As $\Val(E)$ is a Banach space (see (\ref{eq banach Val})), for any $K\subset \mathbf{B}$ we have
\begin{equation*}
  |\phi(K)| = |\phi_k ^+ (K) - \phi_k ^- (K)| \leq \phi_k ^+ (\mathbf{B}) + \phi_k ^- (\mathbf{B}) \rightarrow 0.
\end{equation*}
Hence, $\phi(K) =0$ for any $K\subset \mathbf{B}$, implying $\phi =0$.
\end{proof}

Let us also note that the set of $\mathcal{P}$-positive valuations $\phi_\mu$, where $\mu$ has bounded support, is dense in $\mathcal{P}_i$ with respect to the topology given by $||\cdot||_\mathcal{C}$.

\subsection{Relationship between $\mathcal{V}^\mathcal{P}$ and $\mathcal{V}^\mathcal{C}$}

\begin{prop}\label{prop norm ineq2}
For any $\phi\in \mathcal{V}' _i$, one has that $||\phi||_{\mathcal{P}} \leq C||\phi||_{\mathcal{C}}$ for some uniform constant $C>0$.
Hence, there is a continuous injection:
\begin{equation*}
(\mathcal{V}_i^\mathcal{C},|| \cdot ||_{\mathcal{C}}) \hookrightarrow ( \mathcal{V}_i^\mathcal{P}, || \cdot ||_{\mathcal{P}}).
\end{equation*}
\end{prop}

\begin{proof}
Consider $\phi \in \mathcal{V}_i'$ and assume that $\phi = \phi_+ - \phi_-$ where $\phi_+ , \phi_- \in \mathcal{P}_i$.
Fix some convex bodies $L_1, \ldots , L_i$. One has that:
\begin{equation*}
|\phi(L_1, \ldots , L_i)| \leqslant  |\phi_+(L_1, \ldots , L_i)| + |\phi_-(L_1, \ldots, L_i)|.
\end{equation*}
By Theorem \ref{thrm rev KT} applied to $\phi' = \phi_{\pm}$, $\psi= V(L_1, \ldots, L_i, -[n-i]) $ and to the convex body $\mathbf{B} $ respectively, there exists a uniform constant $C>0$ such that:
\begin{equation*}
\phi_{\pm}(L_1, \ldots , L_i) \leqslant C \phi_\pm (\mathbf{B}) V(\mathbf{B}[n-i], L_1, \ldots, L_i) .
\end{equation*}
In particular, this implies that:
\begin{equation*}
|\phi(L_1, \ldots , L_i)| \leqslant C (\phi_+(\mathbf{B}) + \phi_-(\mathbf{B}))V(\mathbf{B}[n-i], L_1, \ldots, L_i) .
\end{equation*}
By considering two sequences $\phi_{+,j}, \phi_{-,j} \in \mathcal{P}_i$ such that $\lim_j \phi_{+,j}( \mathbf{B}) + \phi_{-,j}(\mathbf{B}) = || \phi ||_{\mathcal{C}}$, we obtain:
\begin{equation*}
|\phi(L_1, \ldots , L_i)| \leqslant C || \phi||_{\mathcal{C}} V(\mathbf{B}[n-i], L_1, \ldots, L_i) ,
\end{equation*}
for any convex bodies $L_1, \ldots, L_i$.
By definition, we obtain:
\begin{equation*}
 || \phi ||_{\mathcal{P}} \leqslant C || \phi ||_{\mathcal{C}},
\end{equation*}
as required.
\end{proof}

\begin{cor} \label{cor injections}
One has the following sequence of continuous injections:
\begin{equation*}
(\mathcal{V}^\mathcal{C},|| \cdot ||_{\mathcal{C}}) \hookrightarrow ( \mathcal{V}^\mathcal{P}, || \cdot ||_{\mathcal{P}}) \hookrightarrow (\Val_i(E) , || \cdot || ) .
\end{equation*}
\end{cor}

\begin{proof}
This follows directly from Proposition \ref{prop norm ineq1} and Proposition \ref{prop norm ineq2}.
\end{proof}


\subsection{An extension of the convolution operator} \label{section extension *}
Recall that for $\phi_\mu \in \mathcal{V}'_i, \psi_\nu \in \mathcal{V}'_j$, we have defined $\phi_\mu * \psi_\nu \in \mathcal{V}'_{i+j-n}$ to be
\begin{align*}
\phi_\mu * \psi_\nu (-)
= \dfrac{i! j!}{n!} \int_{\mathcal{K}(E)^{2n-i-j}} V(-; A_1, \ldots , A_{n-i}, B_1 ,\ldots B_{n-j} ) d\mu(A) d\nu(B).
\end{align*}

Set $\mathcal{V}_i^\mathcal{C}, \mathcal{V}_i^\mathcal{P}$ to be the completions of the space $\mathcal{V}' _i$ with respect to the norms $||\cdot||_\mathcal{C}$ and $||\cdot||_\mathcal{P}$ respectively.

%





We show that the convolution, denoted by $*$, extends continuously to the spaces $\mathcal{V}_i ^\mathcal{C}$ with respect to $||\cdot||_\mathcal{C}$. Note that by Corollary \ref{cor injections}, $\mathcal{V}_j ^\mathcal{C}$ embeds continuously into $\mathcal{V}_j ^\mathcal{P}$, we show that $\mathcal{V} ^\mathcal{P}$ has a continuous module structure over $\mathcal{V} ^\mathcal{C}$ in the following sense.
We now introduce another notation which will simplify the proof of our next result.
Take $\phi, \psi\in \mathcal{V}' _i$, we shall write $\phi \preceq \psi$ (or equivalently, $\psi \succeq \phi$) when
\begin{equation*}
  \phi(L_1,...,L_i) \leq \psi(L_1, ...,L_i)
\end{equation*}
for any $L_1,...,L_i \in \mathcal{K}(E)$.
Observe that it follows from the definition of our notation that $\psi * \phi_1 \leqslant \psi * \phi_2$ when $\psi\in \mathcal{P}_i$ and $\phi_1 \preceq \phi_2$ for $\phi_1,\phi_2 \in \mathcal{V}_j'$.
We also check immediately that the $\GL(E)$ actions preserve the order in the sense that $g\cdot \phi_1 \preceq g\cdot \phi_2$ whenever $\phi_1 \preceq \phi_2$, where $\phi_1, \phi_2 \in \mathcal{V}_i'$ and $g\in \GL(E)$.

\begin{thrm}\label{thrm extension *} The following properties are satisfied.
\begin{enumerate}
\item[(i)]The convolution $*: \mathcal{V}'_i \times  \mathcal{V}'_j \rightarrow \mathcal{V}'_{i+j -n}$ extends continuously to a  bilinear map
\begin{align*}
  *: & \mathcal{V}^\mathcal{C}_i \times \mathcal{V}^\mathcal{C} _j    \rightarrow \mathcal{V}^\mathcal{C} _{i+j -n}\\
  & (\Phi, \Psi)\mapsto \Phi * \Psi.
\end{align*}
Thus $\mathcal{V}^\mathcal{C}$ has a structure of graded Banach algebra with unit given by $\vol$.
\item[(ii)] The convolution $*: \mathcal{V}'_i \times  \mathcal{V}'_j \rightarrow \mathcal{V}'_{i+j -n}$ extends continuously to a  bilinear map
\begin{align*}
  *: & \mathcal{V}^\mathcal{C}_i \times \mathcal{V}^\mathcal{P} _j    \rightarrow \mathcal{V}^\mathcal{P} _{i+j -n}\\
  & (\Phi, \Psi)\mapsto \Phi * \Psi.
\end{align*}
Thus the space $\mathcal{V} ^\mathcal{P}$ has a structure of graded $\mathcal{V}^\mathcal{C}$-module.
\end{enumerate}


\end{thrm}
Before going into the proof of Theorem \ref{thrm extension *}, we shall use the following technical result.
\begin{lem}\label{lem norm submult}
Let $\phi_\mu \in \mathcal{P}_i, \psi_\nu \in \mathcal{P}_j$, then there is $c>0$ depending only on $i, j, n, \vol(\mathbf{B})$ such that:
\begin{itemize}
  \item $||\phi_\mu * \psi_\nu||_\mathcal{C} \leq c ||\phi_\mu ||_\mathcal{C} || \psi_\nu||_\mathcal{C}$;
  \item $||\phi_\mu * \psi_\nu||_\mathcal{P} \leq c ||\phi_\mu ||_\mathcal{P} || \psi_\nu||_\mathcal{P}$.
\end{itemize}
\end{lem}

\begin{proof}
Let us first prove the first inequality. Note that
\begin{equation*}
  ||\phi_\mu * \psi_\nu||_\mathcal{C} = \phi_\mu * \psi_\nu (\mathbf{B})\leq c \phi_\mu (\mathbf{B}) \psi_\nu (\mathbf{B}) = c ||\phi_\mu ||_\mathcal{C} || \psi_\nu||_\mathcal{C},
\end{equation*}
where the second estimate follows from Theorem \ref{thrm rev KT}.

For the second inequality, let $L_1, ..., L_{i+j-n}\in \mathcal{K}(E)$, we have
\begin{align*}
 \phi_\mu * \psi_\nu (L_1, ..., L_{i+j-n}) &= \frac{i!j!}{n!}\int_{\mathcal{K}(E)^{2n-i-j}} V(L_1, ..., L_{i+j-n}, A_1,...,A_{n-i},B_1,...,B_{n-j} )d\mu(A) d\nu(B)\\
 &\leq c||\phi_\mu ||_\mathcal{P} \int_{\mathcal{K}(E)^{n-j}} V(\mathbf{B}[n-i];L_1, ..., L_{i+j-n},B_1,...,B_{n-j} )d\nu(B)\\
 &\leq c||\phi_\mu ||_\mathcal{P}  ||\psi_\nu ||_\mathcal{P} V(\mathbf{B}[2n-i-j];L_1, ..., L_{i+j-n}).
\end{align*}
Thus, by definition $||\phi_\mu * \psi_\nu||_\mathcal{P} \leq c ||\phi_\mu ||_\mathcal{P} || \psi_\nu||_\mathcal{P}$.
\end{proof}

\begin{proof}[Proof of Theorem \ref{thrm extension *}]
We first prove assertion $(i)$.
Assume that $\{\phi_k\} \subset \mathcal{V}_i' , \{\psi_k\} \subset \mathcal{V}_j' $ are Cauchy sequences with respect to the norm $||\cdot||_{\mathcal{C}}$, and $\phi_k\rightarrow \Phi, \psi_k\rightarrow \Psi$.
We show that $\{\phi_k * \psi_k\} \subset \mathcal{V}_{i+j-n}' $ is also a Cauchy sequence with respect to $||\cdot||_{\mathcal{C}}$.

As $\{\phi_k\}, \{\psi_k\}$ are Cauchy sequences, by the definition of the cone norm $||\cdot||_{\mathcal{C}}$, we have the following properties:
\begin{enumerate}
  \item For any $\varepsilon>0$ and for all $k, l$ large enough, there exist decompositions
  \begin{equation*}
  \phi_k - \phi_l = \phi^+ - \phi^-, \ \psi_k - \psi_l = \psi^+ - \psi^-
  \end{equation*}
  such that $\phi^{\pm}\in \mathcal{P}_i, \psi^{\pm}\in \mathcal{P}_j$ and
  \begin{equation*}
  \phi^+ (\mathbf{B})+ \phi^- (\mathbf{B})< \varepsilon, \psi^+ (\mathbf{B})+ \psi^- (\mathbf{B})< \varepsilon.
  \end{equation*}
  \item There exist two decompositions
  \begin{equation*}
  \phi_k = \phi_k ^+ - \phi_k ^-,\ \psi_k = \psi_k ^+ - \psi_k ^-
  \end{equation*}
  such that $\phi_k ^{\pm}\in \mathcal{P}_i, \psi_k ^{\pm}\in \mathcal{P}_j$ and such that
      \begin{equation*}
        \phi_k ^+ (\mathbf{B}) + \phi_k ^- (\mathbf{B})\leq C,\ \psi_k ^+ (\mathbf{B}) + \psi_k ^- (\mathbf{B})\leq C
      \end{equation*}
      for a uniform constant $C>0$.
\end{enumerate}
We write $\phi_k * \psi_k - \phi_l * \psi_l$ as follows:
\begin{equation} \label{eq_dev_1}
\begin{array}{lll}
  \phi_k * \psi_k - \phi_l * \psi_l&= \phi_k * (\psi_k - \psi_l) + (\phi_k - \phi_l)* \psi_l\\
  &=(\phi_k ^+ - \phi_k ^-) * (\psi^+ - \psi^-) + (\phi^+ - \phi^-)* (\psi_k ^+ - \psi_k ^-)\\
  &=(\phi_k ^{+} * \psi^{+}  + \phi_k ^{-} * \psi^{-} + \phi^{+} *\psi_k ^{+} + \phi^{-} * \psi_k^{-})\\
  &\ \ -(\phi_k ^{+} * \psi^{-}  + \phi_k ^{-} * \psi^{+} + \phi^{+} *\psi_k ^{-} + \phi^{-} * \psi_k^{+}).
\end{array}
\end{equation}

This is a decomposition of $\phi_k * \psi_k - \phi_l * \psi_l$ as a difference of two elements in $\mathcal{P}_{i+j -n}$. By Lemma \ref{lem norm submult} applied to each term of \eqref{eq_dev_1}, we get
\begin{equation}\label{eq cauchy1}
  ||\phi_k * \psi_k - \phi_l * \psi_l||_{\mathcal{C}}\leq c' C\varepsilon,
\end{equation}
where $c'$ depends only on $\vol(\mathbf{B}), i, j, n$. Thus $\{\phi_k * \psi_k\}$ must be a Cauchy sequence with respect to the norm $||\cdot||_{\mathcal{C}}$.

Next, assume that $\{{\phi'}_k\}, \{{\psi'}_k\}$ are another two Cauchy sequences also satisfying ${\phi'}_k \rightarrow \Phi, {\psi'}_k \rightarrow \Psi$, we need to verify that the limits of $\{{\phi'}_k *{\psi'}_k\}$ and $\{\phi_k * \psi_k\}$ are the same, i.e.,
\begin{equation*}
  \lim_{k\rightarrow \infty} ||{\phi'}_k *{\psi'}_k - \phi_k * \psi_k||_{\mathcal{C}} =0.
\end{equation*}
Since $||{\phi'}_k -{\phi}_k||_\mathcal{C} \rightarrow 0$ and $||{\psi'}_k -{\psi}_k||_\mathcal{C} \rightarrow 0$, this follows from similar arguments as above.

In particular, the convolution of $\Phi, \Psi$ is defined by the following (well-defined) limit:
\begin{equation*}
  \Phi * \Psi := \lim_{k\rightarrow \infty} \phi_k *\psi_k \in \mathcal{V}^\mathcal{C} _{i+j-n} \subset \Val_{i+j-n}(E).
\end{equation*}

This finishes the proof of assertion $(i)$.

\bigskip
%

Let us prove assertion $(ii)$.
Assume that $\phi_k \in \mathcal{V}' _i$ is a Cauchy sequence in $\mathcal{V}^\mathcal{P}$ converging to $\Phi$ (i.e $\lim_k ||\phi_k - \Phi ||_{\mathcal{P}} =0$). Assume that $\psi_k \in \mathcal{V}' _j$ is a Cauchy sequence in $\mathcal{V}^\mathcal{C}$ converging to $\Psi$ (i.e $\lim_k ||\psi_k - \Psi ||_{\mathcal{C}} =0$). We show that $\phi_k * \psi_k$ is a Cauchy sequence in $\mathcal{V}^\mathcal{P}$.

As $\{\phi_k\}, \{\psi_k\}$ are Cauchy sequences, by the definition of the norms, we have the following properties:
\begin{enumerate}
  \item There exists $c>0$ such that for all $k$ we have
  \begin{align*}
  -c V(\mathbf{B}[n-i];-)\preceq \phi_k  \preceq c V(\mathbf{B}[n-i];-).
  \end{align*}

  \item For any $\varepsilon>0$ and for all $k, l$ large enough,
  \begin{align*}
  -\varepsilon V(\mathbf{B}[n-i];-)\preceq \phi_k - \phi_l\preceq \varepsilon V(\mathbf{B}[n-i];-).
  \end{align*}

  \item There exists $c>0$ such that for all $k$ we have a decomposition $\psi_k = \psi_k ^+ - \psi_k ^-$ such that
  \begin{align*}
   \psi_k ^+ (\mathbf{B}) + \psi_k ^- (\mathbf{B}) \leq c.
  \end{align*}

  \item For any $\varepsilon>0$ and for all $k, l$ large enough, there is a decomposition $\psi_k -\psi_l = \psi ^+ - \psi ^-$ such that
  \begin{align*}
   \psi ^+ (\mathbf{B}) + \psi ^- (\mathbf{B}) \leq \varepsilon.
  \end{align*}
\end{enumerate}

We write $\phi_k * \psi_k - \phi_l * \psi_l$ as follows:
\begin{align*}
  \phi_k * \psi_k - \phi_l * \psi_l = (\phi_k - \phi_l) * (\psi_l ^+ - \psi_l ^-) + \phi_k * (\psi^+ - \psi^-).
\end{align*}
For the first term on the right hand side, as $\psi_l ^{\pm}$ is given by positive measures and since the pointwise inequalities are preserved by convolution with $\mathcal{P}$-positive valuations, we get using the above  estimate:
\begin{align*}
 -\varepsilon V(\mathbf{B}[n-i];-) * \psi_l ^{\pm} \preceq (\phi_k - \phi_l) * \psi_l ^{\pm} \preceq \varepsilon V(\mathbf{B}[n-i];-) * \psi_l ^{\pm}.
\end{align*}
Thus, by Lemma \ref{lem norm submult}, $||(\phi_k - \phi_l) * \psi_l ^{\pm} ||_\mathcal{P} \leq \varepsilon c_1$ for some uniform constant $c_1>0$. Similarly, $||\phi_k* \psi ^{\pm} ||_\mathcal{P} \leq \varepsilon c_2$ for another constant $c_2>0$. This shows that $\phi_k * \psi_k$ gives a Cauchy sequence with respect to the norm $||\cdot||_{\mathcal{P}}$.

Next, assume that $\{{\phi'}_k\}, \{{\psi'}_k\}$ are another two Cauchy sequences also satisfying $\lim_k ||{\phi'}_k - \Phi||_\mathcal{P} =0, \lim_k ||{\psi'}_k - \Psi ||_\mathcal{C} =0$, we need to verify that the limits of $\{{\phi'}_k *{\psi'}_k\}$ and $\{\phi_k * \psi_k\}$ are the same, i.e.,
\begin{equation*}
  \lim_{k\rightarrow \infty} ||{\phi'}_k *{\psi'}_k - \phi_k * \psi_k||_{\mathcal{P}} =0.
\end{equation*}
Since $||{\phi'}_k -{\phi}_k||_\mathcal{P} \rightarrow 0$ and $||{\psi'}_k -{\psi}_k||_\mathcal{C} \rightarrow 0$, this follows from similar arguments as above.

We have thus proved that assertion $(ii)$ holds.
\end{proof}

\section{A variant of Minkowski's existence theorem}\label{sec minkowski}
By the discussion in Example \ref{sec exmple}, the classical Minkowski's existence theorem shows that every strictly $\mathcal{P}$-positive element in $\mathcal{P}_1$ is of the form $V(-;K[n-1])$.
In this section, we discuss a generalization of this result, proving Theorem \ref{thrmx general mink intr}.

\subsection{Existence of the solutions}

Recall that a valuation $\phi \in \mathcal{P}_i$ is called strictly $\mathcal{P}$-positive if its polarized function is pointwisely greater or equal than the polarized function associated to $V(\mathbf{B}[n-i]; \cdot )$.
\begin{thrm}\label{thrm general mink}
For any strictly $\mathcal{P}$-positive valuation $\psi \in \mathcal{P}_i$, there is a constant $c>0$ (depending only on $\psi$) and a convex body $B$ with $\vol(B)=1$ such that
\begin{equation*}
  \psi *V(B [i-1]; -  )=cV(B[n-1];-) \in \Val_1(E).
\end{equation*}
\end{thrm}

In the following proof, we denote by $\phi_B $ the valuation given by $\phi_B= V(B [i-1]; -  [n-i+1])$ where $B$ is a convex body.

Given $\psi\in \mathcal{P}_i$, by scaling the convex set $B$, Theorem \ref{thrm general mink} implies that the functional equation (with unknown $B\in \mathcal{K}(E)$):
\begin{equation*}
  \left(\psi - V(B[n-i];-)\right)*\phi_B =0 \in \Val_1 (E), \ \textrm{where}\ \vol(B)>0
\end{equation*}
always admits a solution.

\begin{proof}
The proof is inspired by the method in \cite{lehmxiao16convexity}\footnote{It was realized in \cite{lehmann2016correspondences} that the same ideas had previously appeared in the classical work of Alexandrov \cite{alexandrov1938}.}. We consider the following variational problem:
\begin{equation*}
  c:=\inf_{M\in \mathcal{K}(E), \vol(M)=1} \psi(M).
\end{equation*}

\bigskip
\textbf{Claim 1}: Let $\{M_l\}$ be a minimizing sequence, that is, $\vol(M_l)=1$ and $\psi(M_l)\searrow c$, then we prove that up to some translations, the sequence $\{M_l\}$ is compact in $(\mathcal{K}(E), d_H)$.
\bigskip

Since $\psi\in \mathcal{P}_i$ is strictly $\mathcal{P}$-positive, there exists an $\epsilon>0$ such that:
\begin{equation*}
\psi(L_1, \ldots ,L_i) \geqslant \epsilon V(\mathbf{B}[n-i], L_1, \ldots , L_i)
\end{equation*}
for any convex body $L_1, \ldots, L_i$. In particular, one has that
\begin{equation*}
V(K [n-i] , M[i]) \leq \psi(M)
\end{equation*}
for any convex body $M$ where $K = 1/\epsilon^{n-i} \mathbf{B}$.
Then there is a uniform constant $d>0$ such that
\begin{equation*}
  V(K[n-i]; M_l[i])\leq d
\end{equation*}
for the minimizing sequence $M_l$.

By Alexandrov-Fenchel's inequality, we have
\begin{equation*}
  V(K[n-i]; M_l[i])\geq V(K[n-1], M_l)^{\frac{n-i}{n-1}} \vol(M_l)^{\frac{i-1}{n-1}} = V(K[n-1], M_l)^{\frac{n-i}{n-1}},
\end{equation*}
where the last equality follows from $\vol(M_l)=1$. In particular, $V(K[n-1], M_l)$ is uniformly bounded above. Let $r_l>0$ be the minimal number such that $M_l \subset r_l K$ (up to a translation). Or equivalently, $1/r_l$ is the maximal number such that $M_l /r_l \subset K$ (up to a translation). By the Diskant inequality,
\begin{align*}
 1/r_l &\geq \frac{V(K[n-1], M_l)^{\frac{1}{n-1}} - \left(V(K[n-1], M_l)^{\frac{n}{n-1}}-\vol(K)\vol(M_l)^{\frac{1}{n-1}}\right)^{\frac{1}{n}}}{\vol(M_l)^{\frac{1}{n-1}}}\\
 &\geq \frac{\vol(K)}{nV(K[n-1], M_l)},
\end{align*}
where the last inequality follows from the generalized binomial formula (see also \cite{lehmann2016correspondences}). We get\footnote{This can also be obtained by applying Theorem \ref{thrm rev KT}.}
\begin{equation}\label{eq diskant}
  r_l \leq nV(K[n-1], M_l)/\vol(K).
\end{equation}
Thus the sequence $r_l$ is uniformly bounded above. Then Blaschke selection theorem implies that, up to translations, the sequence $M_l$ has an accumulation point $B\in \mathcal{K}(E)$ with $\vol(B)=1$. In particular,
\begin{equation*}
  c=\psi(B)=\inf_{M\in \mathcal{K}(E), \vol(M)=1} \psi(M).
\end{equation*}

\bigskip
\textbf{Claim 2}: For any $N\in \mathcal{K}(E)$, we have
\begin{equation}\label{eq mink psef1}
  \frac{n!}{i!(n-i+1)!}\psi*\phi_B (N) - \psi(B) V(B[n-1], N) \geq 0,
\end{equation}
and
\begin{equation}\label{eq mink psef2}
  \frac{n!}{i!(n-i+1)!}\psi*\phi_B (B) - \psi(B) V(B[n-1], B)=0.
\end{equation}
\bigskip

Note that, since the minimal of the variational problem is achieved at $M= B$, for any $t \geq 0$ and any convex body $N$, we have
\begin{equation*}
  \psi\left(\frac{B+tN}{\vol(B+tN)^{1/n}}\right) \geq \psi(B).
\end{equation*}
Calculating the right derivative at $t=0$ implies
\begin{equation*}
\frac{n!}{i!(n-i+1)!}\psi*\phi_B (N) - \psi(B) V(B[n-1], N) \geq 0.
\end{equation*}
The equality (\ref{eq mink psef2}) for $B$ follows from the minimal property of $B$.

\bigskip
\textbf{Claim 3}: There is a convex body $L$ with non-empty interior such that
\begin{equation*}
 \frac{n!}{i!(n-i+1)!}\psi*\phi_B (-) = V(L[n-1],-).
\end{equation*}

By the discussion in Example \ref{sec exmple}, this is a direct consequence of Minkowski's existence theorem since $\psi*\phi_B\in \mathcal{P}_1$ is strictly $\mathcal{P}$-positive.
\bigskip

Now we can finish the proof of our theorem. By Claim 2 and 3, we have
\begin{equation*}
V(L[n-1],N)-\psi(B)V(B[n-1], N)\geq 0
\end{equation*}
for any $N\in \mathcal{K}(E)$. Let $N=L$, we get
\begin{equation*}
  \vol(L) = V(L[n-1],L)\geq \psi(B)V(B[n-1], L) \geq \psi(B)\vol(B)^{\frac{n-1}{n}}\vol(L)^{\frac{1}{n}}.
\end{equation*}
Thus $\vol(L)^{\frac{n-1}{n}}\geq \psi(B)\vol(B)^{\frac{n-1}{n}}$.
On the other hand, let $N=B$, the equality in Claim 2 implies
\begin{equation*}
  V(L[n-1], B)= \psi(B) \vol(B) \geq \vol(L)^{\frac{n-1}{n}}\vol(B)^{\frac{1}{n}}.
\end{equation*}
Thus $V(L[n-1], B) = \vol(L)^{\frac{n-1}{n}}\vol(B)^{\frac{1}{n}}$, which implies that $L= \psi(B)^{\frac{1}{n-1}} B$. Then we get
\begin{equation*}
 \frac{n!}{i!(n-i+1)!}\psi*\phi_B (-) = V(L[n-1],-) = \psi(B)V(B[n-1],-).
\end{equation*}
This finishes the proof of the result.
\end{proof}

\subsection{Compactness of the solution set}
In Minkowski's existence theorem, up to some translation, the solution is unique.
In the generalized case,
we show that the (normalized) solution set of the functional equation (with unknown $B\in \mathcal{K}(E)$)
\begin{equation*}
  \left(\psi - V(B[n-i];-)\right)*\phi_B =0 \in \Val_1 (E), \ \textrm{where}\ \vol(B)=1, \phi_B (-)=V(-;B[i-1]),
\end{equation*}
is compact in $(\mathcal{K}(E), d_H)$.

\begin{prop}\label{prop compct}
Given any strictly $\mathcal{P}$-positive valuation $\psi \in \mathcal{P}_i$, up to translations, the set of normalized solutions of the above equation is compact.
\end{prop}

\begin{proof}
Fix a convex body $L$ with non-empty interior. Since $\vol(B)=1$, similar to the argument in Theorem \ref{thrm general mink}, by Blaschke selection theorem and the Diskant inequality it is sufficient to show that $V(B; L[n-1])$ is uniformly bounded above.

To this end, note that
\begin{equation*}
  V(B[n-1],L)\geq V(B, L[n-1])^{\frac{1}{n-1}} \vol(B)^{\frac{n-2}{n-1}},
\end{equation*}
thus it is sufficient to prove the upper bound for $V(B[n-1],L)$. By the functional equation, we get
\begin{equation*}
  \frac{n!}{i!(n-i+1)!}(\psi*\phi_B)(L) = V(B[n-1],L).
\end{equation*}
Assume that $\psi$ is given by the measure $\mu$, then
\begin{align*}
  \frac{n!}{i!(n-i+1)!}(\psi*\phi_B)(L)& = \int_{\mathcal{K}(E)^{n-i}} V(B[i-1], L; A_1,...,A_{n-i})d\mu(A_1,...,A_{n-i})\\
  &\leq c V(B[i-1], L[n-i+1])\int_{\mathcal{K}(E)^{n-i}} V(L[i], A_1,...,A_{n-i})d\mu(A_1,...,A_{n-i}),
\end{align*}
where the second inequality follows from Theorem \ref{thrm rev KT}, and $c>0$ depends only on $n,i,\vol(L)$.
Then it is sufficient to give a upper bound for $V(B[i-1], L[n-i+1])$.

Since $\psi$ is strictly $\mathcal{P}$-positive,
\begin{equation*}
  1=\vol(B)=\frac{n!}{i!(n-i+1)!}(\psi*\phi_B)(B) \geq c' V(L[n-i];B[i]),
\end{equation*}
thus $V(B[i], L[n-i])$ is uniformly bounded above. On the other hand, since $\vol(B)=1$, the Alexandrov-Fenchel inequality implies that
\begin{equation*}
  V(B[i], L[n-i]) \geq V(B[i-1], L[n-i+1])^{\frac{n-i}{n-i+1}}.
\end{equation*}
Thus $V(B[i-1], L[n-i+1])$ is uniformly bounded above, which implies the compactness of the solution set.
\end{proof}

\begin{rmk}
By the above proof, it is clear that the compactness result holds whenever the $\vol(B)$ has a uniformly positive lower bound.
\end{rmk}

\begin{rmk}
Using the same argument as in Theorem \ref{thrm general mink} and Proposition \ref{prop compct}, one can get an analog statement in complex geometry (see also \cite[Section 5]{lehmxiao16convexity}).

\smallskip
   \begin{quote}
   Let $X$ be a compact K\"ahler manifold of dimension $n$. Assume that $\Theta \in H^{k,k}(X, \mathbb{R})$ is a strictly positive $(k, k)$ class in the sense that for some K\"ahler class $\omega$ the class $\Theta - \omega^k$ contains some positive $(k,k)$ current. Let
      \begin{equation*}
      c = \inf_{A \ \textrm{K\"ahler},\ \vol(A)=1} (\Theta\cdot A^{n-k}).
      \end{equation*}
      Then there is a decomposition
      \begin{equation*}
      \Theta \cdot B^{n-k-1} = c B^{n-1} + \mathcal{N},
      \end{equation*}
     where $B$ is big and nef satisfying $\vol(B)=1$,  $\mathcal{N} \cdot N \geq 0$ for any nef class $N$ and $\mathcal{N} \cdot B =0$.
   Moreover, the set of the (normalized) solutions $B$ is compact.
   \end{quote}
%
\end{rmk}

\section{Linear action on valuations}\label{sec dynamical degree}

In this section, we focus on the linear action on the space of valuations. To do so, we introduce some quantities called dynamical degrees, which we relate to the spectral radius of the linear actions on the space of valuations.

\subsection{On the dynamical degrees}
Recall that $\GL(E)$ has a natural action on $\Val(E)$, which is defined by
\begin{equation*}
  (g\cdot \phi)(K)=\phi(g^{-1}K).
\end{equation*}
The space $\Val_i (E)$ is fixed by this action. Furthermore, by Example \ref{exmple action}, the map $\phi \mapsto g\cdot \phi$ maps the positive cone $\mathcal{P}_i$ to $\mathcal{P}_i$.

\begin{defn}[Degree]
Take $\psi \in \mathcal{P}_i $ and $\phi \in \mathcal{P}_{n-i} $  two strictly $\mathcal{P}$-positive valuations, the \emph{$(n-i)$-th degree} of $g\in \GL(E)$ with respect to $\phi, \psi$ is defined by
\begin{equation*}
  \deg_{n-i} (g) = (g\cdot \psi)*\phi.
\end{equation*}
\end{defn}

We then define the dynamical degree as the asymptotic ration of the sequence $(\deg_{n-i}(g^k))_k$.

\begin{defn}[Dynamical degree]
Given $g\in \GL(E)$ and two strictly $\mathcal{P}$-positive valuations $\phi \in \mathcal{P}_{n-i} $ and $\psi \in \mathcal{P}_i $, the \emph{$(n-i)$-th dynamical degree} of $g$ is defined by
\begin{align*}
  d_{n-i} (g) :&= \lim_{k\rightarrow \infty} \deg_{n-i} (g^k)^{1/k}\\
  & =\lim_{k\rightarrow \infty} ((g^k\cdot \psi)*\phi)^{1/k}.
\end{align*}
\end{defn}

\begin{rmk}
 In the study of the dynamics of a holomorphic map $f : X \to X$ where $X$ is a projective toric variety whose action commutes with the torus action, one can define a notion degree:
 \begin{equation*}
 \deg_i(f) = \int_X f^* \omega^i \wedge \omega^{n-i},
\end{equation*}
where $\omega$ is a K\"ahler class on $X$. In this setting, these two notions coincide and the $i$-degree of $f$ is equal to the $i$-degree of the linear map associated to $f$.
\end{rmk}

Let us explain why the $(n-i)$-th dynamical degree exists, that is, the limit defining $d_{n-i} (g)$ exists, and $d_{n-i} (g)$ is independent of the choices of the two valuations $\psi \in \mathcal{P}_i , \phi \in \mathcal{P}_{n-i}$.
The existence of $d_{n-i} (g)$ follows  directly from the following result.

\begin{lem} Assume that $\phi \in \mathcal{P}_{n-i}$ and $\psi \in \mathcal{P}_{i}$ are given by
$$\psi (-) = V(-;B[n-i]) \in \mathcal{P}_i, \phi (-) = V(-;B[i]) \in \mathcal{P}_{n-i} ,$$
where $B\in \mathcal{K}(E)$ has non-empty interior.
We consider the $n-i$-th degree $\deg_{n-i}$ given by $\phi, \psi$. Assume $f, g \in \GL(E)$, then there is a constant $C>0$ depending only on $\vol(B), n, i$ such that
\begin{equation*}
 \deg_{n-i} (fg) \leq C \deg_{n-i} (f) \deg_{n-i} (g).
\end{equation*}
In particular, given $g\in \GL(E)$, the sequence $\{\log \left(C\deg_{n-i} (g^k)\right)\}_{k=1} ^\infty$ is subadditive, that is,
\begin{equation*}
  \log (C\deg_{n-i} (g^{k+l}))\leq \log (C\deg_{n-i} (g^k)) + \log (C\deg_{n-i} (g^l)),\ \textrm{for any}\ k, l\in\mathbb{N}.
\end{equation*}
\end{lem}

\begin{proof}
For any convex body $B$, let us denote by $\phi_B$ and $\psi_B$ given by $\phi_B = V(B[i] , -[n-i])$ and $\psi_B = V(B[n-i], -[i])$.

Since $\deg_{n-i} (-)$ is given by $\psi_B$ and $\phi_B$, we get
\begin{equation} \label{eq_submul1}
\begin{array}{ll}
  \deg_{n-i} (f) \deg_{n-i}(g)&= \left((f\cdot \psi_B)*\phi_B\right)\left((g\cdot \psi_B)*\phi_B\right)\\
  &=|\det(fg)|^{-1} \left( \psi_{f(B)}*\phi_B\right)\left(\psi_{g(B)}*\phi_B\right)\\
  &= |\det(fg)|^{-1} |\det f|^{-1}
  \left( \psi_{f(B)}*\phi_B\right)
  \left(\psi_{fg(B)}*\phi_{f(B)}\right).
\end{array}
\end{equation}
Note that there exists a constant $c'>0$ such that $(\psi_{f(B)}*\phi_B)
  (\psi_{fg(B)}*\phi_{f(B)}) = c'\phi_B (f(B))\psi_{fg(B)} (f(B))$.
By Theorem \ref{thrm rev KT}, there is a uniform constant $c>0$ such that
\begin{equation} \label{eq_submul2}
\begin{array}{ll}
  (\psi_{f(B)}*\phi_B)
  (\psi_{fg(B)}*\phi_{f(B)}) &\geq c \vol(f(B)) (\psi_{fg(B)}*\phi_{B})\\
  &=c |\det f| |\det fg| \vol(B) ((fg \cdot \psi_B)* \phi_B)\\
  &=c |\det f| |\det fg| \vol(B) \deg_{n-i} (fg).
\end{array}
\end{equation}

Thus, \eqref{eq_submul1} and \eqref{eq_submul2} imply that
\begin{equation*}
\deg_{n-i}(fg) \leq C\deg_{n-i} (f) \deg_{n-i}(g),
\end{equation*}
where $C = 1/ (c \vol(B))>0$ and the inequality is satisfied.
\end{proof}

\begin{rmk}
In the study of complex dynamics, the analogous estimate for rational self-maps is obtained in \cite{dinh2005borne, dinh2004regularization} using the theory of positive currents.  The above simple proof is inspired by \cite{dang2017degrees}.
\end{rmk}

\begin{lem}[Fekete lemma]
For every subadditive sequence $\{a_k\}_{k=1}^\infty$, the limit $\lim_{k\rightarrow \infty}\frac{a_k}{k}$ exists and
\begin{equation*}
\lim_{k\rightarrow \infty}\frac{a_k}{k}=\inf_{k \geq 1} \frac{a_k}{k}.
\end{equation*}
\end{lem}

We now explain why the dynamical does not depend on any choice of strictly positive valuations.
\begin{thrm}\label{thrm dydegree}
Given $g\in \GL(E)$, the dynamical degree $d_{n-i} (g)$ exists and is independent of the choices of strictly $\mathcal{P}$-positive $\psi \in \mathcal{P}_i , \phi\in \mathcal{P}_{n-i} $. In particular, we have
\begin{equation*}
  d_{n-i} (g)=|\det g|^{-1} \lim_{k\rightarrow \infty} (V(g^{k}(B)[n-i], B[i]))^{1/k}.
\end{equation*}
\end{thrm}

\begin{proof}
Recall from Section \ref{sec positive} that we write by $\phi_1 \preceq \phi_2$, if the polarized function associated to $\phi_1$ is pointwisely smaller or equal to the polarized function associated to $\phi_2$.

If $\psi= V(-;\mathbf{B}[n-i]), \phi= V(-;\mathbf{B}[i])$, the existence of $d_{n-i} (g)$ follows directly from Fekete's lemma.

For the independence on $\psi$, $\phi$, we first note that
\begin{equation*}
\psi \preceq || \psi||_{\mathcal{P}} V(\mathbf{B}[n-i]; - ),\ \phi \preceq || \phi||_{\mathcal{P}} V(\mathbf{B}[i]; - ),
\end{equation*}
which follow from the definition of $||\cdot||_{\mathcal{P}}$.
Observe also that the linear action preserves the pointwise inequalities so that:
\begin{equation*}
g^k \cdot \psi \preceq  || \psi ||_{\mathcal{P}} g^k \cdot V(\mathbf{B}[n-i]; -) .
\end{equation*}
Since the convolution with a $\mathcal{P}$-positive preserves each inequality, we obtain:
\begin{equation*}
(g^k \cdot \psi) * \phi \leqslant || \psi ||_{\mathcal{P}}  g^k \cdot V(\mathbf{B}[n-i]; -) * \phi.
\end{equation*}
Then we get:
\begin{equation}\label{eq indep1}
(g^k \cdot \psi )* \phi \leqslant || \psi ||_{\mathcal{P}} || \phi ||_{\mathcal{P}} (g^k \cdot V(\mathbf{B}[n-i];- )) * V(\mathbf{B}[i];- ).
\end{equation}

On the other hand, by the strict positivity of
$\psi, \phi$, there is a constant $C>0$ depending only on $\psi, \phi$ such that
\begin{equation}\label{eq indep2}
  C (g^k \cdot V(\mathbf{B}[n-i], -)* V(\mathbf{B}[i], - ) \leqslant (g^k \cdot \psi)*\phi.
\end{equation}

Thus, the inequalities (\ref{eq indep1}), (\ref{eq indep2}) imply that $d_{n-i}(g)$ does not depend on the choices of $\phi\in \mathcal{P}_{n-i}$ and $\psi \in  \mathcal{P}_i$.
\end{proof}

As a consequence of the above formula, we are able to relate different dynamical degrees as follows.

\begin{prop}\label{prop Logconcave}
For any $g\in \GL(E)$, the sequence $\{d_i (g)\}$ is log-concave, that is, for $1\leq i \leq n-1$
\begin{equation*}
  d_i (g) ^2 \geq d_{i-1} (g) d_{i+1} (g).
\end{equation*}
\end{prop}

\begin{proof}
By Theorem \ref{thrm dydegree}, we get
\begin{align*}
  d_{i}(g) &= \lim_{k\rightarrow \infty} (|\det g^k|^{-1}V(g^{k}(B)[i], B[n-i]))^{1/k}\\
   &= |\det g|^{-1} \lim_{k\rightarrow \infty} (V(g^{k}(B)[i], B[n-i]))^{1/k},
\end{align*}
where $B$ is a fixed convex body with non-empty interior.
Then the log-concavity property follows immediately from the Alexandrov-Fenchel inequality for mixed volumes.
\end{proof}

\subsection{The spectrum of the action on valuations}
Let $g\in \GL(E)$, then by Example \ref{exmple action} it induces a linear operator (denoted by $g_i$):
\begin{equation*}
g_i:\mathcal{V}'_i \rightarrow \mathcal{V}'_i.
\end{equation*}

In the following, let $\gamma\in \{\mathcal{C}, \mathcal{P}\}$.

We first show that $g_i$ extends to a map:
\begin{equation*}
g_i:\mathcal{V}^\gamma _i \rightarrow \mathcal{V}^\gamma _i.
\end{equation*}

\begin{lem}
Let $g\in \GL(E)$. Assume that $||\phi_k -\phi||_\gamma \rightarrow 0$, then $||g\cdot\phi_k -g\cdot\phi||_\gamma \rightarrow 0$.
\end{lem}

\begin{proof}
For the norm $||\cdot||_\mathcal{C}$, it is obvious.

We only need to deal with the norm $||\cdot||_\mathcal{P}$. By definition, we have
\begin{equation*}
  |(\phi_k -\phi)(L_1, ...,L_i)|\leq ||\phi_k -\phi||_\mathcal{P} V(\mathbf{B}[n-i];L_1, ...,L_i),
\end{equation*}
which implies
\begin{align*}
  |g\cdot(\phi_k -\phi)(L_1, ...,L_i)|&\leq ||\phi_k -\phi||_\mathcal{P} V(\mathbf{B}[n-i];g^{-1}(L_1), ...,g^{-1}(L_i))\\
  &=||\phi_k -\phi||_\mathcal{P} \frac{1}{|\det g|}V(g(\mathbf{B})[n-i];L_1, ...,L_i).
\end{align*}

On the other hand, by Theorem \ref{thrm rev KT} we have
\begin{equation*}
  V(g(\mathbf{B})[n-i];L_1, ...,L_i)\leq c V(g(\mathbf{B})[n-i]; \mathbf{B}[i]) V(\mathbf{B}[n-i];L_1, ...,L_i),
\end{equation*}
where $c>0$ depends only on $n, i, \vol(\mathbf{B})$. Hence,
\begin{equation*}
  ||g\cdot(\phi_k -\phi)||_\mathcal{P} \leq c\frac{1}{|\det g|} V(g(\mathbf{B})[n-i]; \mathbf{B}[i]) ||\phi_k -\phi||_\mathcal{P}.
\end{equation*}

This finishes the proof of the result.
\end{proof}

Next we show that the dynamical degree $d_{n-i} (g)$ is just the spectral radius of this operator.

\begin{thrm}\label{thrm norm degree}
Let $g\in \GL(E)$ and let $g_i$ be the induced operator on $\mathcal{V}^\gamma _{i}$, then the following equality is satisfied:
\begin{equation*}
  d_{n-i}(g) = || g_{n-i} : \mathcal{V}_i^\mathcal{P}\to \mathcal{V}_i^\mathcal{P} || = || g_{n-i} : \mathcal{V}_i^\mathcal{C} \to \mathcal{V}_i^\mathcal{C} ||,
\end{equation*}
where the symbol $ || g_{n-i} : \mathcal{V}_i^\mathcal{P}\to \mathcal{V}_i^\mathcal{P} ||$ and $ || g_{n-i} : \mathcal{V}_i^\mathcal{C} \to \mathcal{V}_i^\mathcal{C} ||$ denotes the norm of the operator $g_{n-i}$ on $\mathcal{V}_i^\mathcal{P}$ and $\mathcal{V}_i^\mathcal{C}$ respectively.
\end{thrm}

\begin{proof} For simplicity, since each space is endowed with its appropriate norm, we denote by $|| g_{n-i} ||_\mathcal{P}$ and $|| g_{n-i} ||_{\mathcal{C}}$ the norm of the operator $g_{n-i}$ on $\mathcal{V}_i^\mathcal{P}$ and $\mathcal{V}_i^\mathcal{C}$ respectively.
We need to verify the equality
\begin{equation*}
d_{n-i}(g) = \lim_{k\rightarrow \infty} ||g_i^k||_{\gamma} ^{1/k}.
\end{equation*}

We first consider the case when $\gamma = \mathcal{C}$.

Let $\phi_\mathbf{B} = V(\mathbf{B}[n-i];-)$, by definition we get
\begin{align*}
  &||g^k \cdot \phi_\mathbf{B}||_{\mathcal{C}} =(g^k \cdot \phi_\mathbf{B}) (\mathbf{B}) = V(g^k (\mathbf{B})[n-i], \mathbf{B}[i])/|\det g|^k,\\ &||\phi_\mathbf{B}||_{\mathcal{C}} = \phi_\mathbf{B} (\mathbf{B}) = V(\mathbf{B}[n-i], \mathbf{B}[i]).
\end{align*}
This implies that
\begin{equation}\label{eq norm1}
  ||g_i^k||_{\mathcal{C}} \geq \frac{V(g^k (\mathbf{B})[n-i], \mathbf{B}[i])}{|\det g|^k \vol(\mathbf{B})}.
\end{equation}

On the other hand, take a sequence $\phi_l \in \mathcal{P}_{i} - \mathcal{P}_{i}$ such that $||\phi_l||_{\mathcal{C}} =1$ and $||g^k \cdot \phi_l||_{\mathcal{C}} \rightarrow ||g_i^k||_{\mathcal{C}}$ as $l \rightarrow \infty$. For $l_0$ large enough, we have $||g_i^k||_{\mathcal{C}} \leq 2 ||g^k \cdot \phi_{l_0}||_{\mathcal{C}}$. Assume that $\phi_{l_0} = \phi_{l_0} ^+ - \phi_{l_0}^-$ is a decomposition for $\phi_{l_0}$, then
\begin{equation*}
||g_i^k||_{\mathcal{C}} \leq 2(g^k \cdot\phi_{l_0} ^+ (\mathbf{B})+ g^k \cdot\phi_{l_0}^- (\mathbf{B})).
\end{equation*}
For the term $g^k \cdot\phi_{l_0} ^+ (\mathbf{B})$, by Theorem \ref{thrm rev KT} there is a constant $c>0$ depending only on $n, i, \vol(\mathbf{B})$ such that
\begin{align*}
  g^k \cdot\phi_{l_0} ^+ (\mathbf{B})& = \int_{\mathcal{K}(E)^{n-i}} V(g^{-k}(\mathbf{B})[i], A_1,...,A_{n-i})d\mu_{l_0} ^+(A_1,...,A_{n-i})\\
  &\leq c V(g^{-k}(\mathbf{B})[i],\mathbf{B}[n-i]) \int_{\mathcal{K}(E)^{n-i}} V(\mathbf{B}[i], A_1,...,A_{n-i})d\mu_{l_0} ^+(A_1,...,A_{n-i})\\
  &=c V(g^{-k}(\mathbf{B})[i],\mathbf{B}[n-i]) \phi_{{l_0} } ^+ (\mathbf{B}).
\end{align*}
Similarly,
\begin{equation}
g^k \cdot\phi_{l_0} ^- (\mathbf{B})\leq c V(g^{-k}(\mathbf{B})[i],\mathbf{B}[n-i]) \phi_{{l_0} } ^- (\mathbf{B}).
\end{equation}
Since $||\phi_{{l_0}}||_\mathcal{C} =1$, we get
\begin{equation}\label{eq norm2}
  ||g_i^k||_{\mathcal{C}} \leq 2c V(g^{-k}(\mathbf{B})[i],\mathbf{B}[n-i]).
\end{equation}

\bigskip

Next we consider the case when $\gamma=\mathcal{P}$. Note that $||\phi_\mathbf{B}||_\mathcal{P} =1$. By the definition for $||g^k \cdot \phi_\mathbf{B}||_\mathcal{P}$, we have $g^k \cdot \phi_\mathbf{B} (\mathbf{B}) \leq ||g^k \cdot \phi_\mathbf{B}||_\mathcal{P} \vol(\mathbf{B})$, hence
\begin{equation*}
  ||g^k \cdot \phi_\mathbf{B}||_\mathcal{P} \geq V(\mathbf{B}[n-i], g^{-k}(\mathbf{B})[i])/\vol(\mathbf{B}).
\end{equation*}
This implies that
\begin{equation}\label{eq norm3}
  ||g_i^k||_{\mathcal{P}} \geq \frac{V(g^k (\mathbf{B})[n-i], \mathbf{B}[i])}{|\det g|^k \vol(\mathbf{B})}.
\end{equation}

On the other hand, take a sequence $\phi_l$ such that $||\phi_l||_{\mathcal{P}} =1$ and $||g^k \cdot \phi_l||_{\mathcal{P}} \rightarrow ||g_i^k||_{\mathcal{P}}$ as $l \rightarrow \infty$. For $l_0$ large enough, we have $||g_i^k||_{\mathcal{P}} \leq 2 ||g^k \cdot \phi_{l_0}||_{\mathcal{P}}$. For any $L_1,..,L_i$,
\begin{align*}
  |g^k \cdot \phi_{l_0} (L_1,..,L_i)| &= |\phi_{l_0} (g^{-k}(L_1),..,g^{-k}(L_i))|\\
  &\leq ||\phi_{l_0}||_\mathcal{P} V(\mathbf{B}[n-i], g^{-k}(L_1),..,g^{-k}(L_i)).
\end{align*}
Applying $||\phi_{l_0}||_\mathcal{P} =1$ and Theorem \ref{thrm rev KT} yields a uniform constant $c' >0$ such that
\begin{equation}\label{eq norm4}
  ||g_i^k||_{\mathcal{P}} \leq c' V(g^{k}(\mathbf{B})[n-i],\mathbf{B}[i])/|\det g^k|.
\end{equation}

\bigskip

In summary, by (\ref{eq norm1}), (\ref{eq norm2}), (\ref{eq norm3}), (\ref{eq norm4}) and taking the limits, we obtain the desired equality
\begin{equation*}
d_{n-i}(g) = \lim_{k\rightarrow \infty} ||g_i^k||_{\gamma} ^{1/k}.
\end{equation*}
\end{proof}

\subsection{On relative dynamical degrees}
In the study of dynamics of a holomorphic map that preserves some fibration, it is useful to consider a relative version of dynamical degrees. We have a corresponding picture for valuations on convex sets. Let $S$ be a subspace of dimension $m$, and assume that $l: S\rightarrow E$ is the embedding. Assume that $g\in \GL(E)$ fixes the subspace $S$, equivalently, there is a map $f \in \GL(S)$ such that $g\circ l =l\circ f$.

\begin{defn}
Assume that $\psi \in \mathcal{P}_i  (E), \phi \in \mathcal{P}_{n-i+m}  (E)$ are strictly $\mathcal{P}$-positive valuations and consider
 $\tau_B = V(-; B[m])\in \Val_{n-m} (E) $, where $B \in \mathcal{K}(S)$ satisfies $\vol_S (B)>0$, then the $(n-i)$-th relative degree of $g$ is defined by
\begin{equation*}
  \reldeg_{n-i} (g) = (g\cdot \psi)*\phi *\tau_B.
\end{equation*}
\end{defn}

\begin{defn}
The $({n-i})$-th relative dynamical degree of $g$ is defined by
\begin{equation*}
  \reld_{n-i} (g) = \lim_{k \rightarrow \infty} (\reldeg_{n-i} (g^k))^{1/k}.
\end{equation*}
\end{defn}

Similar to Theorem \ref{thrm dydegree}, we have:

\begin{thrm}\label{thrm reldydegree}
The relative dynamical degree $\reld_{n-i} (g)$ exists and is independent of the choices of $\psi \in \mathcal{P}_i , \phi \in \mathcal{P}_{n-i+m} $ (which are strictly $\mathcal{P}$-positive), and $B \in \mathcal{K}(S)$ with non-empty interior.
\end{thrm}

\begin{proof}
The proof is similar to Theorem \ref{thrm dydegree}, so we omit the details. The only extra ingredient is the following reduction formula for mixed volumes (see \cite[Theorem 5.3.1]{schneider_convex}).

\begin{lem}\label{lem reduction mixvol}
Let $k$ be an integer satisfying $1\leq k \leq n-1$, let $H \subset \mathbb{R}^n$ be a $k$-dimensional linear subspace and let $L_1, ...,L_k, K_1, ..., K_{n-k}$ be convex bodies with $L_i \subset H$ for $i=1,...,k$. Then
\begin{equation*}
  \left(\begin{array}{c}n\\k\end{array}\right)
  V(L_1, ...,L_k, K_1, ..., K_{n-k}) = V_{H}(L_1, ..., L_k) V_{H^\perp} (p_{H^\perp}(K_1), ..., p_{H^\perp}(K_{n-k})),
\end{equation*}
where $V_H (\cdot)$ and $V_{H^\perp}(\cdot)$ denote the mixed volume in $H$ and $H^\perp$, and $p_{H^\perp} : \mathbb{R}^n \rightarrow H^\perp$ is the projection map.
\end{lem}

\end{proof}

\begin{rmk}
Similar to the complex geometry setting (see e.g. \cite{dinhRelative}, \cite{dang2017degrees}), one could also establish a product formula between the dynamical degrees and the relative dynamical degrees.
\end{rmk}

\subsection{Relating the dynamical degrees with the eigenvalues of the associated matrix}
In this section, we give a formula for $d_{n-i} (g)$ using the eigenvalues of $g$. The key point is the formula noticed in Theorem \ref{thrm dydegree}:
\begin{equation*}
  d_{n-i} (g)=|\det g|^{-1} \lim_{k\rightarrow \infty} (V(g^{k}(B)[n-i], B[i]))^{1/k}.
\end{equation*}

\begin{thrm}\label{thrm dydeg formula}
Let $g \in \GL(E)$, and assume that $\rho_1, ...,\rho_n$ are the eigenvalues of $g$ satisfying
\begin{equation*}
  |\rho_1|\geq |\rho_2|\geq...\geq |\rho_n|,
\end{equation*}
then the $(n-i)$-th dynamical degree $d_{n-i}(g)=|\det{g}|^{-1} \prod_{k=1} ^{n-i} |\rho_k|$.
\end{thrm}

It is clear that we only need to check the equality
\begin{equation*}
 \widehat{d}_{n-i}(g):= \lim_{k\rightarrow \infty} (V(g^{k}(B)[n-i], B[i]))^{1/k} = \prod_{k=1} ^{n-i} |\rho_k|.
\end{equation*}

\begin{rmk}
In the study of dynamics of monomial maps, the above formula was first obtained in \cite{lin_algebraic_stability_and_degree_growth,favrewulcandegree}. The proof of \cite{lin_algebraic_stability_and_degree_growth} is algebraic, and the proof of \cite{favrewulcandegree} applies some ideas from integral geometry. We present a different (and simpler) approach to the calculation of $d_{i}(g)$, by using positivity results.
\end{rmk}

\subsubsection{Simple case: $d_1 (g)$}
We first discuss the simple calculation for $d_1 (g)$.
We need to verify the formula
\begin{equation*}
\lim_{k \rightarrow \infty} V(g^k (B), B[n-1])^{1/k} = |\rho_1 (g)|.
\end{equation*}

By Theorem \ref{thrm dydegree}, for any $L, M\in \mathcal{K}(E)$ with non-empty interior, we have
\begin{equation*}
  d_1 (g) =|\det g|^{-1}\lim_{k \rightarrow \infty} V(g^k (L), M[n-1])^{1/k}.
\end{equation*}

First, we prove $d_1 (g) \leq |\det g|^{-1}|\rho_1 (g)|$. To this end, we fix a Euclidean structure on $E$ and  assume that $0\in L$ is an interior point. Then for any point $x\in \partial L$ we have $|g(x)|\leq ||g|||x|$, thus
\begin{equation*}
g(L) \subset c ||g|| {\mathbf{B}}
\end{equation*}
where $\mathbf{B}$ is the unit ball and $c = \max_{x\in \partial L} |x|$. In particular, applying the observation to $g^k$ implies
\begin{equation*}
g^k (L) \subset c ||g^k|| \mathbf{B}.
\end{equation*}
Thus,
\begin{align*}
  d_1 (g) &\leq |\det g|^{-1}\lim_{k \rightarrow \infty}  ||g^k||^{1/k} V(c\mathbf{B}, M[n-1])^{1/k}\\
  &=|\det g|^{-1}|\rho_1 (g)|.
\end{align*}

Next, we prove the reverse inequality $d_1 (g) \geq |\det g|^{-1}|\rho_1 (g)|$.
For any $k$, we can take a unit vector $x_k$ such that $|g^k(x_k)|=||g^k||$. We take $L = 2 \mathbf{B}$ and take $M=\mathbf{B}$. Then the segment $S_k :=[0,x_k] \subset L$, yielding
\begin{equation*}
  V(g^k (S_k), M[n-1]) \leq V(g^k (L), M[n-1]).
\end{equation*}
Note that
\begin{equation*}
V(g^k (S_k), M[n-1])=||g^k|| V(||g^k||^{-1}g^k (S_k), M[n-1]).
\end{equation*}
Since $||g^k||^{-1}g^k (S_k)$ is a segment of unit length, Lemma \ref{lem reduction mixvol} implies
\begin{equation*}
V(||g^k||^{-1}g^k (S_k), M[n-1])=n^{-1}V_{g^k (S_k) ^\perp}(M).
\end{equation*}
Since $M = \mathbf{B}$, the volume $V_{g^k (S_k) ^\perp}(M)$ is a constant,
thus
\begin{equation*}
V(g^k (S_k), M[n-1])=c||g^k||.
\end{equation*}
Then taking the limit implies
\begin{align*}
 d_1 (g) &= |\det g|^{-1}\lim_{k \rightarrow \infty} V(g^k (L), M[n-1])^{1/k}\\
 &\geq |\det g|^{-1}\lim_{k \rightarrow \infty} (c||g^k||) ^{1/k}= |\det g|^{-1}|\rho_1 (g)|.
\end{align*}

In summary, we get the formula $d_1 (g) = |\det g|^{-1}|\rho_1 (g)|$.

\subsubsection{General case}
For the general case, the idea is as follows:
\begin{enumerate}
  \item Prove the formula for diagonalizable matrices over $\mathbb{C}$ with distinct eigenvalues;
  \item Show that $d_{n-i}(\cdot)$ is a continuous function over $\GL(E)$;
  \item For an arbitrary $g\in \GL(E)$, approximate $g$ using diagonalizable matrices over $\mathbb{C}$ with distinct eigenvalues and apply the continuity of $d_{n-i}(\cdot)$.
\end{enumerate}

\begin{lem}\label{lem formula diag}
Assume $g \in \GL(E)$ is diagonalizable over $\mathbb{C}$, and assume $g$ has distinct eigenvalues. Then $\widehat{d}_k (g) = \prod_{i=1} ^k |\rho_i (g)|$.
\end{lem}

\begin{proof}
After a change of basis, we could assume that the matrix form of $g$ takes its real Jordan canonical form. Since $g$ has distinct eigenvalues, its real Jordan canonical form can be written as
\begin{equation*}
\left(
  \begin{array}{cccccc}
  J_1 & & & & &\\
      &\ddots & & & &\\
      &  &J_s & & &\\
      & & &\lambda_{s+1} & &\\
      &  & & & \ddots &\\
      &  & & &  & \lambda_n
  \end{array}
\right),
\end{equation*}
where $J_i = \left(\begin{array}{cc} a_i & b_i\\ -b_i & a_i \end{array}\right)$ corresponds to the non-real eigenvalue $\lambda_i = a_i + \sqrt{-1} b_i$, and $\lambda_{s+1},...,\lambda_{n}$ are real eigenvalues.

In order to calculate the dynamical degree of $g$, we consider the following convex body
\begin{equation*}
  K_{\mathbf{r}}=D_{r_1} \times...\times D_{r_s} \times I_{r_{s+1}} \times...\times I_{r_n},
\end{equation*}
where $D_{r_i}$ is a disk of radius $r_i$, and $I_{r_{j}}$ is a segment of length $r_j$ with 0 as its center.

For any $\gamma, \tau\geq 0$, we have $\gamma K_\mathbf{r} + \tau K_\mathbf{t} = K_{\gamma \mathbf{r} + \tau \mathbf{t}}$. In particular, this  gives an explicit formula for $\vol(\gamma K_\mathbf{r} + \tau K_\mathbf{t})$. On the other hand, note that
\begin{equation*}
\vol(K_{\gamma \mathbf{r} + \tau \mathbf{t}})=  \vol(\gamma K_\mathbf{r} + \tau K_\mathbf{t}) = \sum_k \frac{n!}{k!(n-k)!} V(K_\mathbf{r} [k], K_\mathbf{t} [n-k]) \gamma^k \tau^{n-k}.
\end{equation*}
By comparing the coefficients, we get the explicit formula for $V(K_\mathbf{r} ^k, K_\mathbf{t} ^{n-k})$ for any $\mathbf{r}, \mathbf{t}$. Here, we omit the detailed computations.

Next we take $\mathbf{r}=\mathbf{t}=(1,...,1)$ and compute $V(g^p(K_\mathbf{1}) [k], K_\mathbf{1} [n-k])$. To this end, we note that $$g^p(K_\mathbf{1}) = K_{\mathbf{r}_p},$$
where $\mathbf{r}_p = (|\lambda_1|^p,...,|\lambda_s|^p, |\lambda_{s+1}|^p,...,|\lambda_{n}|^p)$.
Then a direct computation shows that
\begin{equation*}
\widehat{d}_k (g) = \prod_{i=1} ^k |\rho_i (g)|.
\end{equation*}
\end{proof}

\begin{rmk}
The calculations in Lemma \ref{lem formula diag} are inspired by the calculations in \cite[Section 5.1]{favrewulcandegree}, where the authors did the computations for diagonalizable maps over $\mathbb{R}$ and also gave a remark for diagonalizable maps over $\mathbb{C}$.
\end{rmk}

Next we show that the dynamical degree function
\begin{equation*}
d_k: \GL(E) \rightarrow \mathbb{R},\ g\mapsto d_k (g)
\end{equation*}
is continuous.

\begin{thrm}\label{thrm continuity}
The dynamical degree $d_k (\cdot)$ is a continuous function on $\GL(E)$. More precisely, let $\{g_l\}_{l \geq 1}, g\in \GL(E)$ endowed with the topology induced by the $L^2$-norm of $E\times E$, then
\begin{equation*}
\lim_{g_l \rightarrow g} d_k (g_l) =d_k (g).
\end{equation*}
\end{thrm}

\begin{proof}
It is sufficient to prove that
\begin{equation*}
\lim_{g_l \rightarrow g} \widehat{d}_k (g_l) =\widehat{d}_k (g).
\end{equation*}

By Theorem \ref{thrm dydegree}, the dynamical degree is independent of the choices of $\phi, \psi$. In the following we take $K = \mathbf{B}$ to be the unit ball with 0 as its center. We have
\begin{equation*}
 \widehat{d}_k (g) = \lim_{k\rightarrow \infty} V(g^p(K)[k],K[n-k])^{1/p}.
\end{equation*}

We first prove $\lim_{l \rightarrow \infty} \widehat{d}_k (g_l) \geq \widehat{d}_k (g)$.
We consider the \emph{inradius} of $g_l ^p (K)$ relative to $g^p (K)$, which is defined by
\begin{equation*}
  r(g_l ^p (K), g^p (K)):=\max \{\lambda>0|\ \lambda g^p (K) \subset g_l ^p (K)\ \ \textrm{up to some translation}\}.
\end{equation*}
Applying the Diskant inequality to $g_l ^p (K), g^p (K)$, we get
\begin{align*}
 r(g_l ^p (K), g^p (K))
 \geq \frac{\vol(g_l ^p (K))}{nV(g_l ^p (K)[n-1], g^p (K))}.
\end{align*}

We next estimate the mixed volume $V(g_l ^p (K)[n-1], g^p (K))$. Note that
\begin{align*}
  V(g_l ^p (K)[n-1], g^p (K))&= |\det(g_l)|^p V(K[n-1], (g_l ^{-p} \circ g^p) (K))\\
  &=|\det(g_l)|^p
  \int_{\mathbb{S}^{n-1}} h_{(g_l ^{-p} {\circ} g^p)(K)}(x) dS(K^{n-1};x),
\end{align*}
where $h_{(g_l ^{-p} {\circ} g^p)(K)}$ is the support function of the convex body $(g_l ^{-p} \circ g^p) (K)$, and $dS(K^{n-1};\cdot)$ is the surface area measure. For any linear map $A: E \rightarrow E$, by the definition of support function we have
\begin{equation*}
  h_{A(K)}(x)=\max\{x\cdot y|\ y\in A(K)\} = \max\{A^T x\cdot y|\ y\in K\}.
\end{equation*}
Thus $h_{A(K)}(x) = h_K (A^T x)$. Since $K = \mathbf{B}$, we have $h_K (x)=h_{\mathbf{B}}(x)=|x|$. Then we get
\begin{align*}
  h_{(g_l ^{-p} {\circ} g^p)(K)}(x) = h_{K}((g_l ^{-p} {\circ} g^p)^T x) &= |(g_l ^{-p} {\circ} g^p)^T x|\\
  &\leq ||g_l ^{-p} {\circ} g^p|| |x| = ||g_l ^{-p} {\circ} g^p|| h_K (x).
\end{align*}

Applying the above inequality to $V(g_l ^p (K)[n-1], g^p (K))$ implies
\begin{equation*}
  V(g_l ^p (K)[n-1], g^p (K)) \leq |\det(g_l)|^p (||g_l ^{-p} {\circ} g^p||) \vol(K).
\end{equation*}
Then we have
\begin{align*}
 V(g_l ^p (K)[k], K[n-k])^{1/p}
  &\geq r(g_l ^p (K), g^p (K))^{k/p}  V(g ^p (K)[k], K[n-k])^{1/p}\\
  & \geq \left(\frac{\vol(g_l ^p (K))}{nV(g_l ^p (K)[n-1], g^p (K))}\right)^{k/p}V(g ^p (K)[k], K[n-k])^{1/p}\\
  & \geq \left(\frac{|\det(g_l)|^p \vol(K)}{n|\det(g_l)|^p (||g_l ^{-p} {\circ} g^p||) \vol(K)}\right)^{k/p}V(g ^p (K)[k], K[n-k])^{1/p}\\
  & = (||g_l ^{-p} {\circ} g^p||^{1/p})^{-k} n^{-k/p}V(g ^p (K)[k], K[n-k])^{1/p}.
\end{align*}

\begin{lem} For any sequence $g_l$ converging to $g$, we have
\begin{equation*}
\lim_{ l \rightarrow + \infty} \lim_{p\rightarrow + \infty} ||g_l ^{-p} {\circ} g^p||^{1/p}\leq 1.
\end{equation*}
\end{lem}

\begin{proof}
We only need to consider the action of $g_l ^{-p} {\circ} g^p$ on invariant subspaces. Assume that $||x|| =1$ and $x \in \ker( g - \lambda I)^{b}$, where $b$ is the multiplicity of the eigenvalue $\lambda$.
By assumption, we have that:
\begin{equation*}
g^{p} (x)  \in \ker( g - \lambda I)^b.
\end{equation*}

By considering the Jordan form of $g$, there exists a constant $C>0$ (independent of $x$, as $||x||=1$) such that:
\begin{equation*}
|| g^{p}(x) || \leqslant C p^b |\lambda|^{p}.
\end{equation*}
Since $g_l$ converges to $g$, $g^{p}(x)$ is in the union of invariant subspaces of $g_l$ which correspond to the eigenvalues converging to $\lambda$. Thus for any fixed $\delta>0$, there exists $l_\delta$ such that when $l\geq l_\delta$, we have
\begin{equation*}
 || g_l^{-p}\circ g^{p}(x) || \leqslant C' p^{b'} (|\lambda|-\delta)^{-p}|\lambda|^{p},
\end{equation*}
where $C', b'$ are uniform constants by considering Jordan forms.
Taking the limits gives
\begin{equation*}
\lim_{l \rightarrow + \infty} \lim_{p \rightarrow + \infty} ||g_l^{-p}\circ g^{p}||^{1/p}  \leqslant 1.
\end{equation*}
\end{proof}

Using the above lemma, we get $\widehat{d}_k(g) \leq \lim_{l\rightarrow \infty} \widehat{d}_k (g_l)$.

Similarly, by studying the inradius of $g^p(K)$ relative to $g_l ^p (K)$, we get $\widehat{d}_k(g) \geq \lim_{l\rightarrow \infty} \widehat{d}_k (g_l)$.
This finishes the proof of the continuity.
\end{proof}

\begin{rmk}
The complex analog of Theorem \ref{thrm continuity} implies the following interesting continuity result for dynamical degrees of holomorphic maps:
\smallskip
   \begin{quote}
   Let $X$ be a compact K\"ahler manifold of dimension $n$. Assume that $f_l, f$ are dominated holomorphic self-maps of $X$, and assume that the induced actions $$f_l ^*, f^*: H^{1,1}(X, \mathbb{R})  \rightarrow H^{1,1}(X, \mathbb{R})$$
   satisfy $\lim_{l\rightarrow \infty} f_l ^* = f^*$, then $\lim_{l\rightarrow \infty} d_k( f_l) = d_k(f)$ holds for any $k$.
   \end{quote}
\smallskip
To our knowledge, the previous result is that: if the induced actions on $H^{k,k}(X, \mathbb{R})$ satisfies $\lim_{l\rightarrow \infty} f_l ^* = f^*$, then $\lim_{l\rightarrow \infty} d_k( f_l) = d_k(f)$.
\end{rmk}

Now we can finish the proof of Theorem \ref{thrm dydeg formula}.

\begin{proof}[Proof of Theorem \ref{thrm dydeg formula}]
It is sufficient to prove $\widehat{d}_k (g) = \prod_{i=1} ^k |\rho_i (g)|$.
Assume that $f \in \GL(E)$ is diagonalizable over $\mathbb{C}$ and has distinct eigenvalues. For any fixed $g\in \GL(E)$, we consider the path $$g_t := (1-t)f + t g.$$
By linear algebra (see e.g. \cite{diagmatrix}), there is a sequence $g_l$ such that each $g_l$ has distinct eigenvalues (thus it is diagonalizable over $\mathbb{C}$), and $\lim_{l \rightarrow \infty} g_l =g$. (Note that this density statement is not true for diagonalizable matrices over $\mathbb{R}$.)

Since the eigenvalues depend continuously on the entries of a matrix, we get $\lim_{l \rightarrow \infty} |\rho_i (g_l)| = |\rho_i (g)|$. Applying $\widehat{d}_k (g_l)= \prod_{i=1} ^k |\rho_i (g_l)|$ and Theorem \ref{thrm continuity} yields
\begin{equation*}
  \widehat{d}_k (g) =\lim_{l \rightarrow \infty} \widehat{d}_k (g_l) = \lim_{l \rightarrow \infty} \prod_{i=1} ^k |\rho_i (g_l)| = \prod_{i=1} ^k |\rho_i (g)|.
\end{equation*}
\end{proof}

\subsection{Invariant valuations by linear actions}\label{sec invariant}



To motivate the discussions, we first recall some facts from complex dynamics. Let $X$ be a compact K\"ahler manifold of dimension $n$, and let $f \in \Aut(X)$ be a holomorphic automorphism of $X$. Positive invariant classes and invariant currents play an important role in the study of dynamics of $f$. We consider the following positive cone in $H^{k, k}(X, \mathbb{R})$:
\begin{equation*}
  \mathcal{P}_k = \{\{\Theta\} \in H^{k, k}(X, \mathbb{R})|
  \ \Theta \ \textrm{is a smooth positive}\ (k,k)\ \textrm{form}\}.
\end{equation*}
It is clear that $\mathcal{P}_k$ is convex. We denote its closure in $H^{k, k}(X, \mathbb{R})$ by $\overline{\mathcal{P}}_k$. It is clear that $\overline{\mathcal{P}}_k$ is a closed convex cone with non-empty interior, satisfying $\overline{\mathcal{P}}_k -\overline{\mathcal{P}}_k =H^{k, k}(X, \mathbb{R})$. Since $f^*$ preserves $\overline{\mathcal{P}}_k$, the Perron-Frobenius theorem implies that there exists an eigenclass $\Gamma_k \in \overline{\mathcal{P}}_k\setminus \{0\}$ such that
\begin{equation*}
f^* \Gamma_k =d_k \Gamma_k,
\end{equation*}
where $d_k$ is the spectral radius of $f^*$ on $H^{k, k}(X, \mathbb{R})$. Moreover, $d_k$ is equal to the $k$-th dynamical degree of $f$ (see e.g. \cite{dinh_sibony_green_currents}).
In the particular case of monomial maps ( \cite[Section 6, 7]{favrewulcandegree}), Favre-Wulcan constructed a $d_i(f)$-invariant class provided that $d_i (f) ^2 > d_{i+1} (f) d_{i-1} (f)$. Our goal is to prove an analog statement in convex geometry, but under a weaker condition.

 Coming back to the study of linear actions on valuations, we now prove a general result (see Theorem \ref{thrm pos invariant}) which will imply Theorem \ref{thrmx pos invariant intro}.

Fix $g \in \GL(E), \phi\in \Val_{n-i}(E)$, we say that $\phi$ is \emph{invariant} (or \emph{$d_i (g)$-invariant}) if $g\cdot \phi = d_{i}(g) \phi$. We observe that the sequence $d_i(g)$ is log-concave, so in particular:
\begin{equation*}
  d_i (g) ^2 \geq d_{i+s} (g) d_{i-s} (g)
\end{equation*}
whenever both $i + s$ and $i-s$ are well defined.

When the previous inequality is strict, the next Theorem proves that the $d_i(g)$-eigenvector are more constrained.
\begin{thrm}\label{thrm pos invariant}
Assume that $2i\leq n$, and $g\in \GL(E)$.
Then the following properties are satisfied.
\begin{enumerate}
\item The subspace of $d_i(g)$-invariant valuations in $\Val_{n-i}(E)$ is non trivial.
\item Assume that
the strict log-concavity inequality is satisfied for $s \leqslant \min(i, n-i)$:
\begin{equation*}
  d_{i}(g)^2 > d_{i-s}(g) d_{i+s}(g),
\end{equation*}
 then
%
   for any two $d_i(g)$-invariant  valuations $\psi_1 \in \mathcal{V}_{n-i} ^\mathcal{P}, \psi_2 \in \mathcal{V}_{n-i} ^\mathcal{C}$ we  have
\begin{equation*}
  \psi_1 * \psi_2  =0.
\end{equation*}

\item Assume that
\begin{equation*}
d_1^2(g) > d_2(g),
\end{equation*}
then there exists a unique (up to a multiplication by a positive constant) $d_1(g)$-invariant valuation $\psi\in \overline{{\mathcal{P}}_{n-1}}$ in the closure of ${\mathcal{P}}_{n-1}$ in $\mathcal{V}_{n-1}^\mathcal{P}$. Moreover, $\psi$ lies in an extremal ray of $\overline{\mathcal{P}_{n-1}}$.
\end{enumerate}
\end{thrm}

\begin{proof}
Let us prove statement $(1)$. Let
us first consider the following two special cases.
\begin{enumerate}
\item[(a)] The matrix of $g$ in the canonical basis has Jordan form and the only eigenvalue of $g$ is $\rho \in \mathbb{R}$.
\item[(b)] One has that $n = 2$ and $g = \rho \Id \circ h$ where $h$ is in the orthogonal group and where $\rho \in ] 0 ,+ \infty [$.
\end{enumerate}

Suppose we are in the case $(a)$. Fix $i \leqslant n $. Let $(e_1, \ldots, e_n)$ be the canonical basis of $E$, let $B$ be the unit ball in $E$ and denote by $E_i = \Vect(e_1, \ldots , e_i)$.
Consider $B_i:=B \cap E_i$ and consider the valuation given by:
\begin{equation*}
\phi_i ( L):= V( B_i[i], L[n-i]).
\end{equation*}
Let us compute $g \cdot \phi_i(L)$ for $L \in \mathcal{K}(E)$:
\begin{equation*}
g \cdot \phi_i(L) = V( B_i[i], g^{-1}(L) [n-i]) = \dfrac{1}{|\det(g)|} V(g(B_i) [i], L [n-i]).
\end{equation*}
By the projection formula for mixed volumes (Lemma \ref{lem reduction mixvol}), since $B_i$ is contained in a subspace of dimension $i$ and since $g$ leaves the subspace $E_i$ invariant, we have:
\begin{equation*}
g \cdot \phi_i(L) := \dfrac{1}{\rho^n}
\left(\begin{array}{c}n\\i\end{array}\right)^{-1}
\vol_{E_i}(g(B_i)) \vol_{E_i^\perp}(p_i(L)),
\end{equation*}
where $p_i: E \to E_i^\perp$ is the orthogonal projection onto $E_i^\perp$.
Since $|\det(g_{|E_i})| = \rho^i $, we have that:
\begin{equation*}
g \cdot \phi_i = \rho^{i-n} \phi_i,
\end{equation*}
as required.
\smallskip

Suppose we are in the case $(b)$. Then $g = \rho \Id \circ h$ where $h$ is an element of the orthogonal group. If $i = 0 $ then the valuation $\vol$ is $\rho^2$-invariant and if $i = 2$, then the trivial valuation constant equal to $1$ is $\rho^0$-invariant.
Let us find a valuation in $\mathcal{P}_1$ which is $\rho$-invariant.
There exists a  ball $K$ in $E$ such that $h( K) = K$. Consider the valuation $\phi \in \mathcal{P}_1$ given by:
\begin{equation*}
\phi(L) := V(K, L),
\end{equation*}
for any $L\in \mathcal{K}(E)$.
We have that:
\begin{equation*}
g \cdot \phi(L) = \dfrac{1}{\rho^2} V(g(K) ,L ) = \dfrac{1}{\rho^2} V( \rho(K) , L) = \dfrac{1}{\rho} V(K,L) = \rho^{-1} \phi(L),
\end{equation*}
as required.
\smallskip

Let us show how statement $(1)$ follows from the previous arguments. Up to a conjugation by an element of $\GL(E)$, we are reduced to the problem of finding a $\rho^{i-n}$-invariant valuation in $\mathcal{P}_{n-i}$ for $0 \leqslant i \leqslant n$, where $ \rho$ is the spectral radius of $g$ in each of the above special cases. Take $g \in \GL(E)$.
By construction, there exists a decomposition of $E $ into:
\begin{equation*}
E = \oplus E_k,
\end{equation*}
where each $E_k$ is a  $g$-invariant subspace such that $g_{|E_k}$ satisfies condition $(a)$ or $(b)$.
Denote by $\lambda_k = \rho(g_{|E_k})$.
On each subspace, there exists a convex body $B_k \subset E_k$ such that the valuation given by $V(B_k [j], - [\dim E_k - j])$ is $\lambda_k^{j - \dim E_k}$-invariant.
Considering a well-chosen valuation of the form
\begin{equation*}
\phi(L) = V(B_1 [i_1], \ldots , B_{k} [i_k], L [n-i]),
\end{equation*}
where $i_1 + \ldots + i_k = i$, gives the required invariant valuation.

%

\bigskip
Let us prove statement $(2)$. First note that it is sufficient to prove
\begin{equation*}
  \psi_1 * \psi_2 * \phi_B =0,
\end{equation*}
where $\phi_B (-)=V(-;B[n-2i])$ and $B\in \mathcal{K}(E)$ is a convex body with non-empty interior and smooth boundary.

Note also that if $\psi \in \Val_{n-i} (E)$ is $d_i (g)$-invariant, then for any $c\neq 0$ and $K\in \mathcal{K}(E)$ we have:
\begin{align*}
  ((c g)\cdot \psi)(K) &= \psi ((c g)^{-1}(K))= c^{i-n}\psi(g^{-1}(K))\\
  &=c^{i-n} (g\cdot \psi)(K)=c^{i-n} d_i (g) \psi(K)\\
  & = d_i (c g) \psi(K),
\end{align*}
thus $(c g)\cdot \psi = d_i (c g) \psi$. In particular, $\psi$ is $g$-invariant if and only if it is $c g$-invariant. Without loss of generality, to simplify the notations, we can assume that $|\det g|=1$.

\bigskip

We take approximations $\psi_{1, l}$ such that $\lim_{l\rightarrow \infty}||\psi_{1, l} - \psi_1 ||_\mathcal{P} =0, \lim_{l\rightarrow \infty}||\psi_{2, l} - \psi_2 ||_\mathcal{C} =0$. Then there is some uniform constant $c>0$ such that for all $l$
\begin{equation}\label{eq bd decomp}
  \psi_{2, l}^+ (\mathbf{B}) + \psi_{2,l} ^- (\mathbf{B}) \leq c,
\end{equation}
and for $l$ large enough
\begin{equation*}
-c V(\mathbf{B}[i];-) \preceq \psi_{1, l} \preceq c V(\mathbf{B}[i];-).
\end{equation*}
As $\psi_{2,l} ^{\pm}$ are given by positive measures, we get
\begin{equation*}
  -c V(\mathbf{B}[i];-)*\psi_{2,l} ^{\pm} \preceq \psi_{1, l}* \psi_{2,l} ^{\pm}\preceq c V(\mathbf{B}[i];-)*\psi_{2,l} ^{\pm},
\end{equation*}
which in turn implies that
\begin{equation}\label{eq bd conv}
  -c g^k \cdot (V(\mathbf{B}[i];-)*\psi_{2,l} ^{\pm})*\phi_B \leq g^k \cdot (\psi_{1, l}* \psi_{2,l} ^{\pm})*\phi_B \leq c g^k \cdot (V(\mathbf{B}[i];-)*\psi_{2,l} ^{\pm})*\phi_B
\end{equation}

Since $\psi_1, \psi_2$ are invariant valuations, we have
\begin{align*}
   g^k\cdot(\psi_1 * \psi_2) * \phi_B &= (g^k\cdot\psi_1) * (g^k\cdot\psi_2) * \phi_B\\
   &= d_i (g) ^{2k} \psi_1 * \psi_2 * \phi_B.
\end{align*}
The decomposition of $\psi_{2, l}$ gives:
\begin{align*}
 g^k \cdot(\psi_1 * \psi_2 )* \phi_B &= \lim_{l\rightarrow \infty}  g^k \cdot(\psi_{1,l} * \psi_{2, l}) * \phi_B\\
  &=\lim_{l\rightarrow \infty}  g^k \cdot \left ( \psi_{1,l} * \psi_{2,l}^+   - \psi_{1,l} * \psi_{2,l}^- \right  ) * \phi_B.
\end{align*}
By (\ref{eq bd conv}), we get:
\begin{equation}\label{eq bound1}
d_i(g)^{2k} (\psi_1 * \psi_2 * \phi_B) \leqslant c \liminf_{l\rightarrow \infty} (g^k \cdot (V(\mathbf{B}[i];-)*\psi_{2,l} ^{+})*\phi_B + g^k \cdot (V(\mathbf{B}[i];-)*\psi_{2,l} ^{-})*\phi_B).
\end{equation}

Applying Theorem \ref{thrm rev KT} to $\phi :=g^k \cdot (V(\mathbf{B}[i];-)*\psi_{2,l} ^{\pm})$, $\psi:= \phi_B \in \Val_{2i}(E)$ and the convex body $K:= g^k(B)$, we obtain
\begin{equation}\label{eq bound kt}
 \vol(g^k (B))\left(g^k \cdot (V(\mathbf{B}[i];-)*\psi_{2,l} ^{\pm}) * \phi_B\right) \leq g^k \cdot (V(\mathbf{B}[i];-)*\psi_{2,l} ^{\pm})(g^k(B)) \phi_B (g^k (B)).
\end{equation}
On the other hand, by Theorem \ref{thrm rev KT} again, we have
\begin{align*}
  \phi_B (g^k (B)) &= V(g^k(B)[2i], B[n-2i])\\
  &=V(g^k(B)[i+s],g^k(B)[i-s], B[n-2i])\\
  &\leq C_1 V(g^k(B)[i+s], B[n-i-s]) V(g^k(B)[i-s], B[n-i+s]),
\end{align*}
where $C_1>0$ is a constant which depends only on $B$, $i$ and $n$.

Note that by (\ref{eq bd decomp}), there is some $C_2 >0$ such that
\begin{equation}\label{eq bound 2}
g^k \cdot (V(\mathbf{B}[i];-)*\psi_{2,l} ^{\pm})(g^k(B)) = V(\mathbf{B}[i];-)*\psi_{2,l} ^{\pm})(B)\leq C_2.
\end{equation}

By (\ref{eq bound1}), (\ref{eq bound kt}), (\ref{eq bound 2}) and the estimate for $\phi_B (g^k (B))$, we deduce that there exists a uniform constant $C_3>0$ such that
\begin{equation}\label{eq bound C4}
  d_i (g) ^{2k} (\psi_1 * \psi_2 * \phi_B)\leq {C_3} V(g^k(B)[i+s], B[n-i-s]) V(g^k(B)[i-s], B[n-i+s]).
\end{equation}

In summary, if $\psi_1 * \psi_2 * \phi_B >0$, after taking $k$-th root of the above inequality (\ref{eq bound C4}) and letting $k$ tend to infinity, we get
\begin{equation*}
  d_i (g) ^2 \leq d_{i+s} (g) d_{i-s} (g).
\end{equation*}
This contradicts with our assumption. Thus,
\begin{equation*}
\psi_1 * \psi_2 * \phi_B \leq 0.
\end{equation*}

Since the valuations $- \psi_1$ is also invariant, the previous argument holds and we also have:
\begin{align*}
(-\psi_1) * \psi_2 * \phi_B \leq 0.
\end{align*}
Hence, we must have $\psi_1 * \psi_2 * \phi_B =0$.
\bigskip

Finally we prove statement $(3)$.  Take $i=n-1$ and suppose that $d_1 (g) ^2 > d_2 (g)$.

We denote by $\overline{\mathcal{P}_{n-1}}^{\mathcal{P}}$ and $\overline{\mathcal{P}_{n-1}}^{\mathcal{C}}$   the closure of $\mathcal{P}_{n-1}$ in $\mathcal{V}^\mathcal{P}$ and $\mathcal{V}^\mathcal{C}$ respectively.
We claim that
\begin{equation*}
\overline{\mathcal{P}_{n-1}}^{\mathcal{P}} =\overline{\mathcal{P}_{n-1}}^{\mathcal{C}}=\mathcal{P}_{n-1}
\end{equation*}
and that any valuation $\phi \in \overline{\mathcal{P}_{n-1}}^{\mathcal{P}}$ is of the form  $V(L; -[n-1])$ for some $L \in \mathcal{K}(E)$.

Let us prove that the inclusion $\overline{\mathcal{P}_{n-1}}^\mathcal{P} \subset \mathcal{P}_{n-1}$ is satisfied. Fix a valuation $\phi \in \overline{\mathcal{P}_{n-1}}^{\mathcal{P}}$, by Corollary \ref{cor density n-1}, there exists a sequence of valuations $\phi_j = V(L_j , -[n-1])$
such that $||\phi_j - \phi||_\mathcal{P} \rightarrow 0$. This implies  that $V(L_j, \mathbf{B}[n-1])$ is uniformly bounded above. By Diskant's inequality (similar to the estimate (\ref{eq diskant})), the convex bodies $L_j$ (up to some translations) are bounded.
We can thus extract a subsequence of $L_j$ (up to some translations) converging to a convex body $L$. In particular, $\phi = V(L, - [n-1])$ as required.

\medskip

We prove that  $\overline{\mathcal{P}_{n-1}}^\mathcal{C} \subset \mathcal{P}_{n-1}$.
Take a valuation $\psi \in \overline{\mathcal{P}_{n-1}}^{\mathcal{C}}$, we explain why $\psi$ can be written as $\psi (-) = V(L; - [n-1])$ for a convex body $L$.
As $\mathcal{P}_{n-1} \subset \overline{\mathcal{P}_{n-1}} ^\mathcal{P}$, any valuation in $\mathcal{P}_{n-1}$ is of the form $V(L;-)$.
Hence there exists  a sequence of convex bodies $L_k \in \mathcal{K}(E)$ such that
\begin{equation*}
||V(L_k; - [n-1]) - \psi||_\mathcal{C} \rightarrow 0.
\end{equation*}
This implies that $V(L_k; \mathbf{B}[n-1])$ is uniformly bounded above. Then the same argument as in the previous step shows that $\psi = V(L; -[n-1])$ for some $L\in \mathcal{K}(E)$, as required.




This finishes the proof of the claim since all the other inclusions are directly satisfied.

\bigskip

Now we have $\psi_1(-)=V(-[n-1]; K)$ and $\psi_2 (-)= V(-[n-1]; L)$ for some $K, L \in \mathcal{K}(E)$. Then statement $(2)$ implies that $\psi_i * \psi_j * \phi_B =0$ for $i, j\in \{1, 2\}$.
In particular, this gives
\begin{equation*}
  V(K, L, B[n-2])=V(K[2], B[n-2])=V(L[2], B[n-2])=0.
\end{equation*}
Hence,
\begin{equation*}
   V(K, L, B[n-2])=V(K[2], B[n-2])V(L[2], B[n-2]).
\end{equation*}
Now the uniqueness result follows from \cite[Theorem 7.6.8]{schneider_convex}, which we present below as a lemma.
\begin{lem}
If the equality holds in
\begin{equation*}
   V(K, L, C_1,...,C_{n-2})\geq V(K[2], C_1,...,C_{n-2})V(L[2], C_1,...,C_{n-2}),
\end{equation*}
where $C_1,...,C_{n-2}$ are smooth convex bodies with non-empty interior, then $K, L$ are homothetic.
\end{lem}
As in our setting, $B$ is smooth, this immediately proves the uniqueness of invariant valuations.

The proof of the extramality of any $\mathcal{P}$-positive $d_1(g)$-invariant valuation follows from the above lemma. Indeed, assume by contradiction that $\psi\in \overline{\mathcal{P}_{n-1}}^\gamma$ is $d_1(g)-$invariant and can be written as
\begin{equation*}
  \psi = \phi_1 +\phi_2,
\end{equation*}
where $\phi_1=V(-; K_1), \phi_2=V(-; K_2)$. We need to verify that $\phi_1, \phi_2$ are proportional. The vanishing of $\psi*\psi *\phi_B$ is equivalent to
\begin{equation*}
   V(K_1, K_2, B[n-2])=V(K_1[2], B[n-2])=V(K_2[2], B[n-2])=0,
\end{equation*}
and the Lemma implies that $K_1, K_2$ are homothetic, hence contradicts our assumption. Thus $\psi$ must lie on an extremal ray of the cone $\overline{\mathcal{P}_{n-1}}^\gamma \subset \mathcal{V}_{n-1} ^\gamma$, where $\gamma = \mathcal{C}$ or $\gamma = \mathcal{P}$.
\end{proof}

\subsubsection{Weak closedness}

The above argument for Theorem \ref{thrm pos invariant} (3) shows that
the cone $\mathcal{P}_{n-1}$ is closed with respect to the topology given by $||\cdot||_\mathcal{P}$. Actually, this cone is also weakly closed in the following sense. Observe that for any convex body $K \in \mathcal{K}(E)$, the evaluation map induces a continuous linear form on $\mathcal{V}_{k}^{\mathcal{P}}$:
\begin{equation*}
  \ev_K : \mathcal{V}_{k}^{\mathcal{P}} \rightarrow \mathbb{R}, \phi \mapsto \phi(K).
\end{equation*}
The continuity of $\ev_K$ follows from
\begin{equation*}
 |\phi(K)| \leqslant ||\phi||_\mathcal{P} V(\mathbf{B}[n-k], K[k]).
\end{equation*}
The weak topology is the coarsest topology on $\mathcal{V}_k^\mathcal{P}$ such that the evaluation maps $\ev_K$ are continuous.
We first note that the weak topology contains a countable basis of neighborhoods.
Consider the finite intersection of neighborboods of the form:
\begin{equation*}
U =  \left \{ \phi \in \mathcal{V}_{k}^\mathcal{P} \ | \  |\phi (P) - \sum_{i=1} ^N a_i V(P_{1,i}, \ldots , P_{n-k, i}, P [k] ) | <   b \right  \},
\end{equation*}
where $a_i, b \in \mathbb{Q}$, $N \in \mathbb{N}$ and where $P$ and $P_{j,i}$ are rational polytopes in $E$.
By construction $U$ is an open set of $\mathcal{V}_k^\mathcal{P}$ for the weak topology.
The fact that such  subset $U$ defines a basis of neighborhoods results from the density of rational polytopes inside $\mathcal{K}(E)$.

\begin{prop}
The cone $\mathcal{P}_{n-1}\subset \mathcal{V}_{n-1}^\mathcal{P}$ is closed with respect to the weak topology.
In particular, one has the following equality:
\begin{equation*}
\overline{\mathcal{P}_{n-1}}^\mathcal{P} = \mathcal{P}_{n-1} = \overline{\mathcal{P}_{n-1}}^w,
\end{equation*}
where $\overline{\mathcal{P}_{n-1}}^w$ is the  closure of the cone $\mathcal{P}_{n-1}$ with respect to the weak topology and where $\overline{\mathcal{P}_{n-1}}^\mathcal{P}$ is the closure of the cone $\mathcal{P}_{n-1}$ with respect to the norm $||\cdot ||_\mathcal{P}$.
\end{prop}

\begin{proof}
Since the space $\mathcal{V}_{n-1}^\mathcal{P}$ endowed with the weak topology is first countable, every point $\phi \in \mathcal{V}_{n-1}^\mathcal{P}$ in the weak closure of the cone $\mathcal{P}_{n-1}$ is the weak limit of a sequence $\phi_j \in\mathcal{P}_{n-1}$.
%
Recall that every valuation in $\mathcal{P}_{n-1}$ is of the form $V(M;-[n-1])$ for some convex body $M\in \mathcal{K}(E)$ and one can write each $\phi_j$ as  $\phi_j = V(L_j; -[n-1])$ where $L_j \in \mathcal{K}(E)$.
Since $\phi_j$ converges weakly to $\phi$, this implies that:
\begin{equation*}
\phi_j (\mathbf{B}) = V(L_j, \mathbf{B}[n-1])  \rightarrow \phi(\mathbf{B}),
\end{equation*}
as $j $ tend to $+ \infty$.
In particular, the sequence $\{V(L_{j}, \mathbf{B}[n-1])\}_{j}$ is bounded.
By Diskant's inequality, there exists a subsequence of the sequence $\{L_{j}\}_{j \in \mathbb{N}}$ (up to translations), which converges to a convex body $L$.
Hence, we have that $\phi = V(L; -[n-1]) $ for some $L \in \mathcal{K}(E)$ and $\phi \in \mathcal{P}_{n-1}$ as required.
\end{proof}

\begin{rmk}
We are not sure about the weak closedness of $\mathcal{P}_k$ when $k \neq n-1$.
\end{rmk}

\begin{rmk}
In general the invariant valuations are not smooth. The invariant valuations in \cite{favrewulcandegree} are given by the volume of a projection onto a linear subspace. By the reduction formula for mixed volumes, they are given by mixed volumes, which are elements in $\mathcal{P}_i$.
\end{rmk}

\begin{rmk}
For any $g\in \GL(E)$, the action of $g$ satisfies $g(\mathcal{P}_i) \subset \mathcal{P}_i$. Recall that in functional analysis we have the famous Krein-Rutman theorem:
\smallskip
\begin{quote}
Let $X$ be a Banach space, and let $\mathcal{C} \subset X$ be a closed convex cone such that $\mathcal{C}-\mathcal{C}$ is dense in $X$. Let $T: X\rightarrow X$ be a non-zero \emph{compact} operator satisfying $T(\mathcal{C})\subset \mathcal{C}$, and assume that its spectral radius $\rho(T)$ is strictly positive. Then there is an eigenvector $u \in \mathcal{C}\setminus \{0\}$ such that $T(u)=\rho(T) u$.
\end{quote}
\smallskip
If $X$ is of finite dimension, then this is the Perron-Frobenius theorem, which is very useful to construct invariant classes in complex dynamics. In our setting, in general the induced linear operator by $g$ is not compact, thus we cannot apply it directly. 
\end{rmk}

\begin{rmk}
We remark that the same vanishing result also holds true for the dynamics of dominated holomorphic maps. Furthermore, by Hodge theory (see e.g. \cite{voisinHodge1}), the extremal ray property holds true for invariant $(1, 1)$ classes. More precisely, using the notations in the previous sections, we have:
\smallskip

   \begin{quote}
   Let $X$ be a compact K\"ahler manifold of dimension $n$. Let $f:X\rightarrow X$ be a dominated holomorphic self-map. Assume $2k \leq n$. If $d_k ^2 > d_{k+s} d_{k-s}$, then for any K\"ahler class $\omega$ and any invariant positive classes $\Theta_1, \Theta_2 \in \overline{\mathcal{P}}_k \subset H^{k,k}(X, \mathbb{R})$ we have
      \begin{equation*}
        \Theta_1 \cdot \Theta_2 \cdot \omega^{n-2k}=0.
      \end{equation*}
      Moreover, if $d_1 ^2 >d_2$, then the non-zero invariant class $\Theta \in \overline{\mathcal{P}}_1$ is unique (up to some scaling) and lies in an extremal ray of $\overline{\mathcal{P}}_1$.
   \end{quote}

\smallskip
The proof of $ \Theta_1 \cdot \Theta_2 \cdot \omega^{n-2k}=0$ is the same as in Theorem \ref{thrm pos invariant}, where we apply the analog of Theorem \ref{thrm rev KT} in complex geometry.
To prove that the invariant nef class  is extremal and unique, we take two nef classes $\Theta_1, \Theta_2$ which we decompose into
\begin{equation*}
  \Theta_i = a_i \omega + P_i,
\end{equation*}
where $a_i \in \mathbb{R}$, and $P_i$ is a primitive class, i.e., $\omega^{n-1} \cdot P_i =0$. Since $\Theta_i \cdot \Theta_j \cdot \omega^{n-2}=0$ for $i, j \in \{1, 2\}$, both $P_1$ and $P_2$ can not be zero. Moreover, combining with $\omega^{n-1} \cdot P_i =0$ implies
\begin{align*}
  &P_1^2 \cdot \omega^{n-2}=-a_1 ^2 \omega^n, \\
  &P_2^2 \cdot \omega^{n-2}=-a_2 ^2 \omega^n,\\
  &P_1\cdot P_2 \cdot \omega^{n-2}=-a_1 a_2 \omega^n.
\end{align*}
Thus the matrix $[P_i\cdot P_j \cdot \omega^{n-2}]_{i, j}$ is degenerate. By Hodge-Riemann bilinear relations, we have $P_1 = cP_2$ for some non-zero constant $c$. Then we get $a_1 ^2 = c^2 a_2 ^2$. We claim $a_1 =c a_2$, which then implies $\Theta_1 = c \Theta_2$. If some $a_i =0$, then this is clear; otherwise, if $a_1 =-c a_2$, by considering $\Theta_1 - c \Theta_2$ we get that $\omega$ is also an invariant class, which is impossible by the vanishing result. Thus we finish the proof of the uniqueness result. The extremity property follows from the same argument.
\end{rmk}

\bibliography{reference}
\bibliographystyle{amsalpha}

\bigskip

\bigskip

\noindent
\textsc{Stony Brook University, Stony Brook, USA}\\
\noindent
\verb"Email: nguyen-bac.dang@stonybrook.edu"

\bigskip

\noindent
\textsc{Tsinghua University, Beijing, China}\\
\noindent
\verb"Email: jianxiao@tsinghua.edu.cn"

\end{document}